\pdfoutput=1
\documentclass[a4paper]{amsart}	

\usepackage[utf8]{inputenc}
\usepackage[T1]{fontenc}
\usepackage{textcomp}
\usepackage{amsmath, amssymb}
\usepackage{amsthm}
\usepackage{mathrsfs}
\usepackage{xcolor}
\usepackage{import}
\usepackage{xifthen}
\usepackage{pdfpages}
\usepackage{transparent}

\newtheorem{problem}{Problem}[section]
\newtheorem{theorem}[problem]{Theorem}
\newtheorem{question}[problem]{Question}

\newtheorem{lemma}[problem]{Lemma}
\newtheorem{corollary}[problem]{Corollary}
\theoremstyle{definition}
\newtheorem{definition}[problem]{Definition}
\newtheorem*{remark}{Remark}
\usepackage{latexsym}
\usepackage{amsmath}
\usepackage{amsfonts}
\usepackage{mathrsfs}

\usepackage{textcomp}
\usepackage{anyfontsize}
\usepackage{booktabs}
\usepackage{multirow}
\usepackage{makecell}
\allowdisplaybreaks
\newcommand{\ra}[1]{\renewcommand{\arraystretch}{#1}}

\numberwithin{equation}{section}
\numberwithin{equation}{section}
\begin{document}

\title{Curvature estimates for stable minimal surfaces with a common free
boundary
}
\author{Gaoming Wang}

\maketitle
\begin{abstract}
The minimal surfaces meeting in triples with equal angles along
a common boundary naturally arise from soap films and other
physical phenomenon. They are also the natural extension of 
the usual minimal surface. In this paper, we consider the
multiple junction surface and show the Bernstein's Theorem 
still holds for stable multiple junction surface in some
special case. The key part is to derive the $L^p $ estimates
of the curvature for multiple junction surface.

\end{abstract}

\section{Introduction}

In the paper of Schoen, Simon, Yau \cite{schoen1975curvature}, 
they'd showed the
$L^p$ curvature estimates for the minimal hypersurfaces. As a corollary, they
could
get the generalized Bernstein's theorem. That is, the only stable immersed
hypersurface $\Sigma^n \hookrightarrow \mathbb{R} ^{n+1}$ ($\text{dim} \Sigma
^n=n$ and $n\le 5$) with area growth condition
(i.e. $\mathcal{H}^{n}(B^{\Sigma }_R)\le CR^n$ for all the intrinsic ball
$B^{\Sigma }$ with radius $R$) is a hyperplane. In \cite{colding2002estimates},
Colding, Minicozzi showed
the stable and $2$-sided, simply connected minimal
surface in $\mathbb{R}^3 $ has
quadratic area growth. So the only stable, complete, 2-sided minimal
surface in $\mathbb{R}^3 $ is a plane. Thus, considering the triple junctions
appearing in the nature phenomenon such as soap film, we have the
following nature problem like the usual generalized Bernstein's Theorem.
\begin{problem}
	\label{Main_Problem}
	If three 2-sided minimal surfaces with boundary meet at the same boundary and each
	two of them meet at exactly $120$ degrees along $\Gamma$. Suppose they are
	complete under the distance function and stable in some suitable variations,
	is it true that each of them is flat?
\end{problem}

The surfaces with triple junctions have been studied extensively due to nature
phenomena. J.Taylor \cite{taylor1976structure} proved certain
locally area-minimizing surfaces should have two types of singularities.
The first one is the 
$Y $-type singularity, which we're interested in. Further more, G. Lawlor and
F. Morgan \cite{lawlor1996curvy} have showed the triple junction surfaces
are always locally minimizing area in any arbitrary dimension and 
codimension. In the result of C. Mese and S. Yamada \cite{mese2006parameterized},
they have reproduced soap films with $Y $-type singularities studied by 
Taylor \cite{taylor1976structure} by minimizing energy with a suitable
boundary condition. So it's nature to extend other properties of 
minimal surfaces to the case of minimal triple junction surface. 
Following from the Schoen's rigidity theorem 
for catenoids \cite{schoen1983uniqueness},
J. Bernstein and F. Maggi \cite{bernstein2021symmetry}
showed the rigidity of $Y $-shaped catenoid (the left figure in 
Fig.\ref{fig:YCatenoid}).

Following from Allard's regularity \cite{allard1972first}, L. Simon 
\cite{simon1993cylindrical} showed if a stationary integral 2-varifold
has density $\frac{3}{2} $ at one point, then it looks like
the $C^{1,\mu} $ triple junction minimal surface. Recently, 
B. Krummel \cite{krummel2014regularity} showed higher regularities
along their triple junctions for stationary integral varifolds.

Besides considering the minimal triple junction surface, one can move
triple junction surface by mean curvature like the usual mean curvature
flow for surfaces. For examples, A. Freire \cite{freire2010mean} (graph case) 
and D. Depner, H. Garcke, et al. 
\cite{depner2013linearized,depner2014mean} (general case) have
considered the mean curvature flow with triple junctions. F. Schulze
and B. White \cite{schulze2020local} showed the local regularity for mean curvature flow
with triple edges. Note that mean curvature flow of curves with
triple junctions in $\mathbb{R}^2  $
is just the usual network flow. There are relatively more
results on this direction, see for examples \cite{mantegazza2004motion,bronsard1993three,ilmanen2014short,tonegawa2016blow}.

In this paper, we will want to extend the curvature estimates
for stable minimal surfaces
to the case of minimal triple junction surface
to see if we can get the Bernstein type theorem like minimal surfaces.
Instead of triple junctions, we can consider arbitrary number
of surfaces meet at the same boundary.

\begin{theorem}
	Suppose $M= (\theta _1\Sigma _1, \cdots ,\theta _q\Sigma _q; \Gamma)$ is the
	(orientable) minimal multiple junction surface
	in $\mathbb{R}^3 $. We assume $M$ is complete, 
	stable and
	has quadratic area growth. Furthermore, we assume $\Gamma$ is compact and
	the angles between $\Sigma _i, \Sigma _j$ keep same along $\Gamma$ for 
	$1\le i,j\le q$.
	Then each $\Sigma_i$ is flat.
\end{theorem}

The terminology will be explained in Section 2.

The special case is the
triple junction surface, $M=(\Sigma _1, \Sigma _2, \Sigma
_3; \Gamma)$ with $\Gamma$ compact. Note the angles between $\Sigma _i, 
\Sigma _j$ are $\frac{2\pi}{3}$ for all $i\neq j$. So from the above theorem, 
we know the $Y$-shaped catenoid is non-stable in our sense. 
\begin{figure}[ht]
 \centering
 \includegraphics[width=0.8\textwidth]{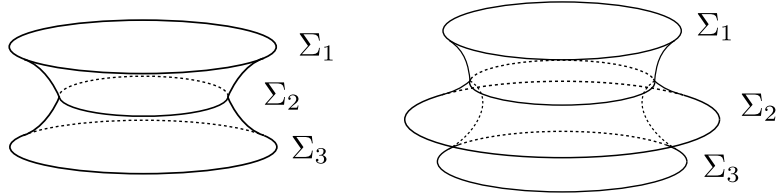}
 \caption{Two kinds of $Y$-shaped catenoid}
 \label{fig:YCatenoid}
\end{figure}

Another case is the $Y$-shaped bent helicoid. Like the usual
bent helicoid, one can construct $Y $-shaped bent helicoid
by the classical Bj\"orling's formula(see \cite{meeks2007bending}
for example). One can choose three unit normal vector
fields making equal angles with each other instead of
one along a unit circle in the construction.
It has a circle as the triple junction
and hence it is not stable.

\begin{figure}
\begin{center}
	\includegraphics[scale=0.1]{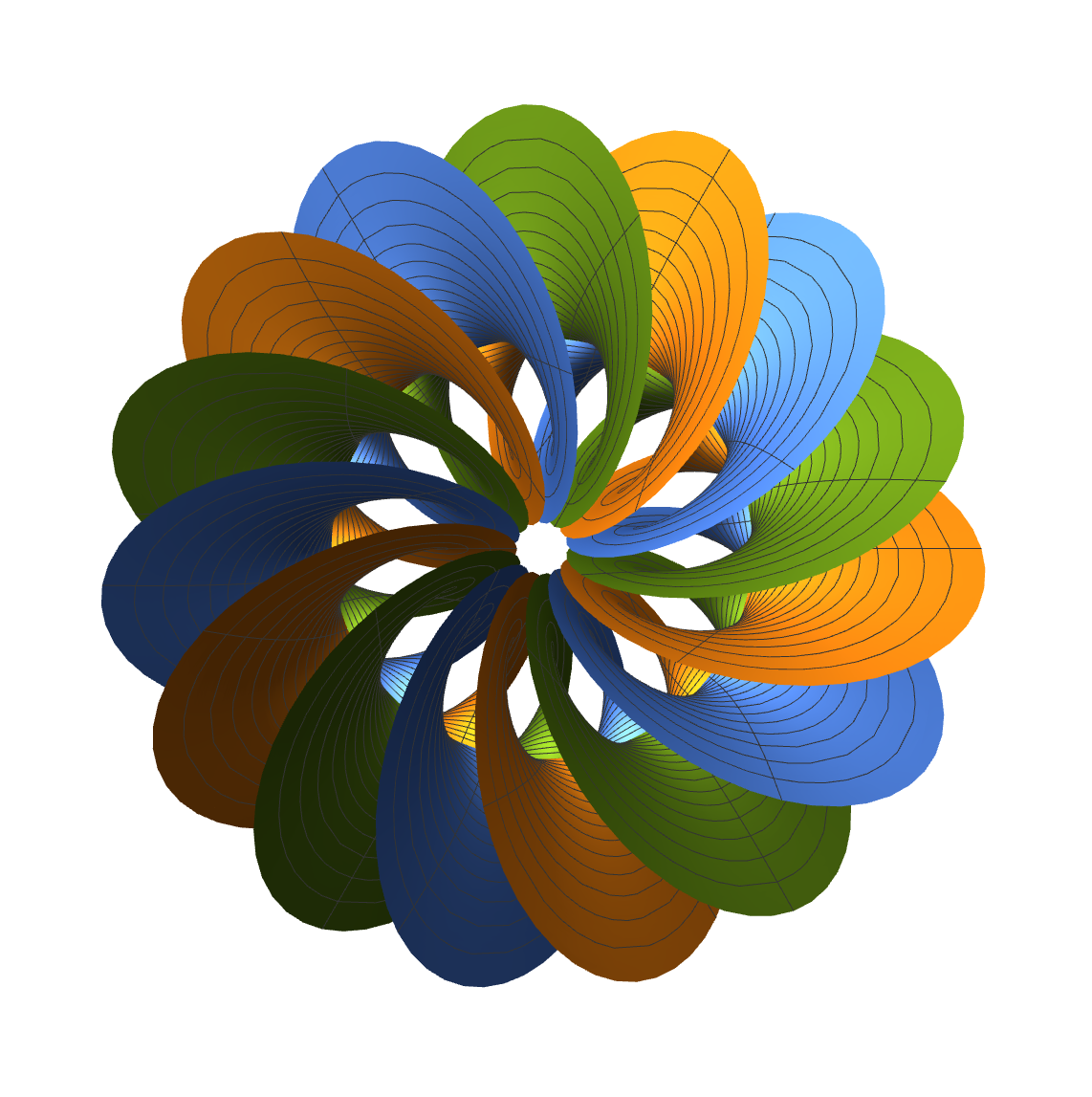}
	\includegraphics[scale=0.14]{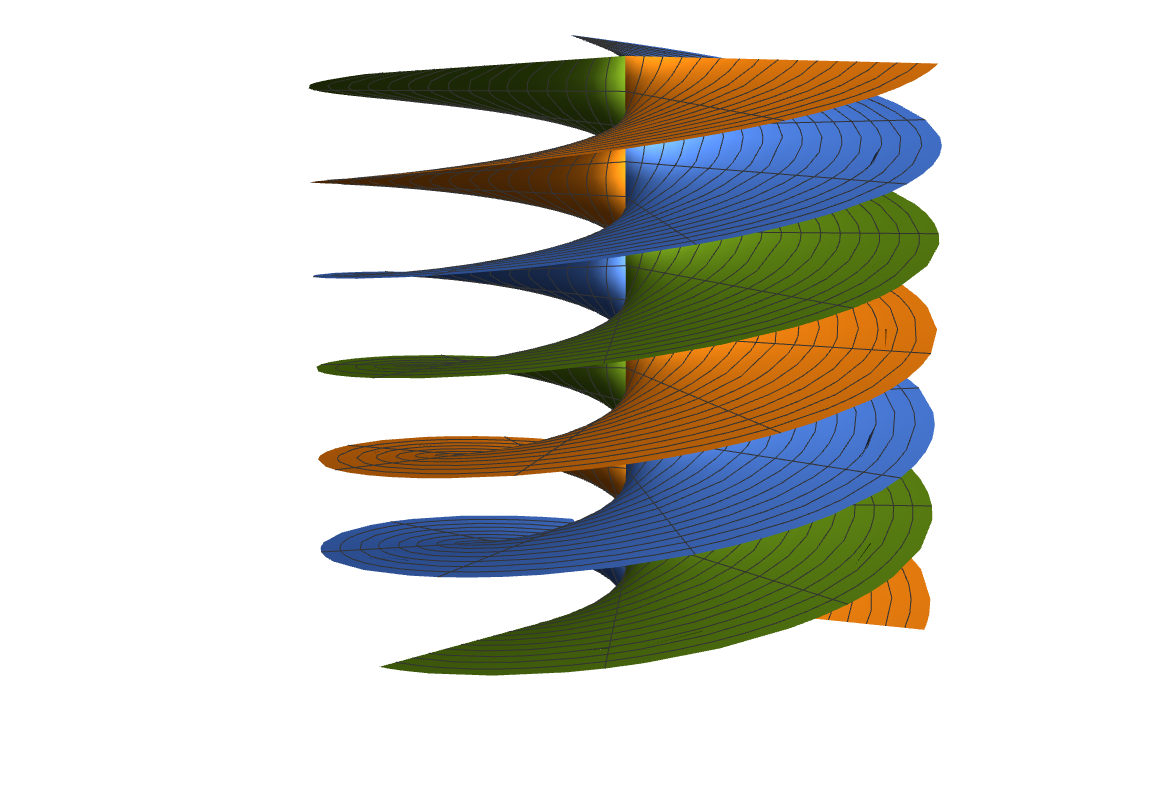}
\end{center}
\caption{$Y$-shaped bent helicoid and usual helicoid}
\label{fig:Y-helicoid}
\end{figure}

For the case if $\Gamma$ is a straight line, we also have similar result.

\begin{theorem}
	Suppose $M$ has the same condition with Theorem \ref{Main_Problem} except
	$\Gamma$ being a straight line instead of being compact. The each $\Sigma _i$ is
	flat.
	\label{Second_Main_Theorem}
\end{theorem}

Note that this theorem is not enough to show the $Y$-shaped helicoid is
non-stable since it does not have quadratic area growth.

The key step proving the above theorems is the following curvature estimates for
minimal multiple junction surface.

\begin{theorem}
	\label{LpEstimate}
	Suppose $M$ is the minimal multiple junction surface. We assume $M$ is
	stable. Then for the smooth function $\phi _i$ defined
	on $\Sigma _i$ with compact support and satisfying compatible condition
	along $\Gamma$, i.e. $\text{sign}(\phi _i)\left|A_i\right|^{p-1}\left|\phi
	_i\right|^{p}$ will be the projection of a smooth vector field along $\Gamma$ to the normal 
direction of $\Sigma _i$. Then we have
\begin{align}
	\label{eq:LpEstimate}
	\sum_{i=1}^{q} \int_{\Sigma _i} \theta_i |\phi_i|^{2p}\left| A_i\right| ^{2p}\le& \sum_{i=1}^{q}
	\theta_i\left[ \int_{\Sigma _i} C \left|\nabla \phi _i\right|^{2}
	\left|A_i\right|^{2p-2}\left|\phi _i\right|^{2p-2}\right.\nonumber\\
	 &\left.\vphantom
																  { \int } +
	\int_{\Gamma} \left( \frac{p-1}{2}\left|\tau_i(\log \left|A_i\right|)
		\right|-\left< \boldsymbol{H}_\Gamma,\tau_i
\right>  \right)\left|A_i\right|^{2p-2}\left|\phi_i \right|^{2p}   \right] 
\end{align}
for $p\in(1,\frac{5}{4})$, where $\boldsymbol{H}_\Gamma$ is the curvature vector of $\Gamma$, $\tau_i$
is the outer conormal of $\Sigma _i$ along $\Gamma$. The constant
$C=C(M)$ doesn't rely on $p$. 
\end{theorem}

The compatible condition will make sure the variation is well-defined on $M$,
see (\ref{eq:compactible_condition}) for details. Note that the integration
of the term
$\left|A_i\right|^{2p-2}\left|\phi _i\right|^{2p}\tau_i(\left|A_i\right|)$ is
still well defined for $\left|A_i\right|=0$ as we'll explain later on.

For the case of triple
junction, i.e. $q=3$ and $\theta _1=\theta _2=\theta _3=1$, this condition is
equivalent to the following identity
\[
	\sum_{i=1}^{3} \text{sign}(\phi _i)\left|A_i\right|^{p-1}\left|\phi
	_i\right|^{p}=0
\]

The proof of (\ref{eq:LpEstimate}) is essentially following Schoen, Simon, Yau's
proof \cite{schoen1975curvature}. Before that, we need to calculate the second
variation formula to get the following stability operator.
\begin{equation}
	\label{eq:Stability}
	\sum_{i=1}^{q} \int_{\Sigma _i} \theta _i \left( \left|\nabla _{\Sigma _i}\phi
	_i\right|^2 -\left|A_i\right|^2\phi _i^2 \right) -\int_{\Gamma} \theta _i
	\phi _i^2 \boldsymbol{H}_\Gamma \cdot \tau_i   
\end{equation}
with $\phi _i$ satisfying compatible condition (\ref{eq:compactible_condition}).

After getting curvature estimates, we can choose a suitable function. The
trick part is we need our functions to satisfy
the compatible conditions. So near $\Gamma$, $\phi _i $ should satisfy some
compatible conditions and the gradient $\phi_i$ cannot
vanish near $\Gamma$. This is why we need $\Gamma$ is compact or $\Gamma$ is a
straight line. So we can choose $p$ close to 1 to control the term near $\Gamma$. For the part that
far from $\Gamma$, we can choose $\phi _i$ like the standard cutoff function in
a large ball. After choosing a suitable function, we can deduce the
curvature needs to vanish everywhere.

\section{Minimal surfaces with multiple junction}
In this section, we will fix some notations and give the definition of the \textit{minimal multiple
junction surface} and several related concepts.

For $q\in \mathbb{N}$, we suppose $\Sigma _1,\cdots ,\Sigma _q$ are all smooth
2-dimensional manifolds with boundary $\partial \Sigma _i$ and $\Gamma$ is a
smooth 1-dimensional manifold. We only consider the case that each $\Sigma _i $
is orientable so we have the well defined unit normal vector field on $\Sigma _i
$ when immersing into $\mathbb{R}^3  $. For each $i$, we suppose there is a
diffeomorphism $p_i:\Gamma\rightarrow \partial \Sigma _i$. We will always
immerse $\Sigma _i $ into $\mathbb{R}^3  $ when talking about extrinsic
geometric quantities like normal vectors, second fundamental form and so on.
The table \ref{tab:notations} lists the notations used in this paper.

\begin{table}[htpb]
	\ra{1.3}
\renewcommand\cellalign{lc}
	\centering
	\caption{Notations}
	\label{tab:notations}
	\begin{tabular}{@{}lll@{}}
		\toprule
	Symbols & $\quad  $ & Meaning\\ \midrule
	$X \cdot Y $ &  & Standard inner product in $\mathbb{R}^3  $.\\
	$D_XY $ &  & Standard coderivative in $\mathbb{R}^3  $.\\ 
	$B_{r}(x) $ &  & Open ball centered at $x $ with
	radius $r $ in $\mathbb{R}^3  $.\\ \midrule
	$T\Sigma \ (\mathrm{resp.}\ T\Gamma) $ &  & Tangent bundle of $\Sigma$ (resp. 
	$\Gamma  $).\\
	$N\Sigma  $ (resp. $N\Gamma  $)&  & Normal bundle of $\Sigma  $
	(resp. $\Gamma  $) in $\mathbb{R}^3  $ .\\
	$\nu_i \in \Gamma (T\Sigma _i)$ &  & Unit normal vector filed on $\Sigma _i $ on $\mathbb{R}^3  $.
	\\
	$\tau_i \in  \Gamma (N
	\partial \Sigma _i)$&  & 
	\makecell{Unit outer conormal of $\partial \Sigma _i $
	on $\Sigma _i $ pointing outside of $\Sigma _i $.\\
	That is, $\tau_i(p) \in T_p\Sigma _i \cap N_p \partial \Sigma _i $
for any $p \in \partial \Sigma _i $.}
	\\
	$A_i(X,Y) =D_XY \cdot \nu_i$ &  & The second
	fundamental form on $\Sigma _i $.\\
	$\left|A_i\right| $ &  & The norm of second fundamental form on
	$\Sigma _i $.\\
	$\boldsymbol{H}_i $ (resp. $\boldsymbol{H}_\Gamma  $) &  & 
	The mean curvature vector of $\Sigma _i $ (resp. $\Gamma  $).\\
	$\mathrm{sign}(x) $	&  & The sign function. \\\bottomrule
	\end{tabular}
\end{table}

%
%
%
%
%
%
%
%
%
%
%

\subsection{Definition of multiple junction surfaces}

\begin{definition}
	We say $M=(\Sigma _1, \cdots ,\Sigma _q;\Gamma)$ is an \textit{intrinsic
	multiple junction surface} if it is a quotient space $\bigcup _{i=1}^q \Sigma
	_i / \!\sim$ where the equivalent relation is defined as the following, 

	$x\sim y$ if and only if $x=y$ or $x\in \partial  \Sigma _i,y\in \partial \Sigma _j$ for some
	$1\le i,j\le q$ and $x=p_i\circ p_j^{-1}(y)$.
	\label{def:quo_mul_jun}
\end{definition}

%

\begin{remark}
	We can define $M$ as a topological space with coordinate charts like the
	definition of smooth manifold. 
\end{remark}

\begin{definition}
	We say $M=(\Sigma _1, \cdots , \Sigma _q;\Gamma)$ is a
	\textit{multiple junction surface} in $\mathbb{R}^3 $ if $(\Sigma _1,\cdots ,\Sigma _q;\Gamma)$
	is an intrinsic multiple junction surface and there is a
	map $\varphi :M\rightarrow \mathbb{R}^3 $ such that the
	restriction of $\varphi$ on $\Sigma _i$ is a smooth immersion for each $1\le
	i\le q$. 
	
	We call the map $\varphi $ \textit{smooth immersion} for $M$.
	\label{def:mul_jun}
\end{definition}

Note that we have a nature metric on each $\Sigma _i$ for $1\le i\le q$ by
pulling back the metric on $\mathbb{R}^3 $.

In general, we will consider the multiple junction surface $M=(\Sigma _1,\cdots
,\Sigma _q; \Gamma)$ with constant density $\theta _1,\cdots ,\theta _q>0$
such that we
have constant density function $\theta _i$ on the surface $\Sigma _i$. 
We will write this
surface as $M=(\theta _1 \Sigma _1, \cdots ,\theta _q \Sigma _q;\Gamma)$. So the
associated 2-varifold of $M$ has the form
\[
	V_M:=\sum_{i=1}^{q} \theta _i \left|\Sigma _i\right|.
\]
Here, $\left|\Sigma _i\right|$ denotes the multiplicity one varifold associated
with the surface $\Sigma _i$. Note that we do not require $\theta _i $ to be 
the integers.


\begin{definition}
	We say a multiple junction surface $M=(\theta _1 \Sigma _1, \cdots ,\theta
	_q \Sigma _q; \Gamma)$ is minimal if each $\Sigma _i$ is a smooth minimal
	immersion and on $\Gamma$, we have
	\[
		\sum_{i=1}^{q} \theta _i \tau_i=0.
	\]
	\label{def:minimal}
\end{definition}

\begin{remark}
	By the regularity of B. Krummel \cite{krummel2014regularity}, suppose a
	stationary integral 2-varifold $V$ has the form

	$$V=\sum_{i=1}^{q} \theta _i\left|\Sigma _k\right|$$
	for	distinct $C^{1,\mu} $ embedded hypersurfaces-with-boundary $\Sigma _1, \cdots
	,\Sigma _q$ with a common boundary $\Gamma$ for some $0<\mu<1$. Then for any $Z\in \Gamma$, if
	$T_Z\Sigma _i$ are not the same plane in $\mathbb{R}^3 $, then we can find a
	neighborhood $O_Z$ of $Z$ such that $\Sigma _i$ is smooth and $\Gamma$ is a
	smooth curve in $O_Z$. Indeed, they are all analytic since $\mathbb{R}^3 $ is a real
	analytic manifold.
	
	Moreover, by the regularity of cylindrical tangent cones by L. Simon
\cite{simon1993cylindrical}, if a stationary integral 2-varifold in $U$ has density
$\frac{3}{2}$ at some point $Z\in U$, then near $Z$, $M$ is the varifold
associated with three $C^{1,\mu}$ minimal surface with a common boundary
$\Gamma$ and $\Gamma$ is still a $C^{1,\mu}$ curve for some $0<\mu<1$. So at
least for the triple junction, we can assume much weaker condition on the
above definition.
\end{remark}

For each $\Sigma _i$, we can define the intrinsic distance function $d_i(x,y)$
for $x,y\in \Sigma _i$, which is the length of the shortest geodetic jointing
$x,y$ on $\Sigma _i$. 

So we can define a global distance function $d(x,y)$ for $x\in \Sigma _i,
y\in \Sigma _j$ by
\begin{align*}
	d(x,y) :={} & \inf \left\{ \sum_{k=0}^{l-1} d
	_{i_k}(x_k,x_{k+1}):x_0=x,x_{l+1}=y, x_1,\cdots ,x_l\in \Gamma, \right.  \\
		    & \quad\ \  \left. \vphantom{\sum_{n=1}^{k} } i_0=i,i _{l-1}=j, 1\le
		  i_1,\cdots ,i_{l-2}\le q,\text{ for }l \in \mathbb{N} \right\} .
\end{align*}

Hence, we use $B_r^M(x)=\{y\in M:d(x,y)<r\} $ to denote the intrinsic ball on
$M$. 

Now we can define the distance function with respect to
$\Gamma$ as
\[
	d_\Gamma(x)=\inf_{y\in \Gamma} d(x,y)\text{ for }x\in M.
\]

\begin{definition}
	We say a multiple junction surface $M$ is complete if it is complete in the
	distance function $d(\cdot,\cdot)$. That is, every Cauchy sequence converges
	to some point in $M$ under this distance function.
	\label{def:complete}
\end{definition}
\subsection{Definition of functional spaces on triple junction surfaces}

From now on, we will always assume $M=(\theta _1\Sigma _1, \cdots ,\theta _q\Sigma _q;\Gamma)$ is a complete minimal multiple junction
surface in $\mathbb{R}^3 $.

When we consider the variation on $M$, we will need to consider a kind of vector
field on $M$. So we have the following definition.

\begin{definition}
	We say a map $X(x):M\rightarrow T_x \mathbb{R}^3 \simeq \mathbb{R}^3 $ is a
	$C^k$ vector field on $M$ if each $X|_{\Sigma _i}$ is a $C^k$
	vector field on $\Sigma _i$ for $1\le i\le  q$. We write this
	vector field space as $C^k(M;T\mathbb{R}^3 )$.
	\label{def: CkVector_Fields}
\end{definition}

Note that we do not require the vector field can be jointed smoothly cross the
junction. For example, let $M=(H_1,H_2;\Gamma)$ with $H_1, H_2$ the opposed two
half planes in $\mathbb{R}^3 $ and $\Gamma$ the straight line in $\mathbb{R}^3 $.
Then as an immersion, $M$ can be regarded as a smooth plane in $\mathbb{R}^3 $
but the smooth vector field on $M$ may not smooth on this plane.

Let's consider the space of functions on $M$. The nature definition is to
consider the function on $M$, which write as $f:M\rightarrow \mathbb{R} $ and
say it is $C^k$ if the restriction on each $\Sigma _i$ is $C^k$ up to boundary.

Somehow this function space is not big enough to contain the function we are
interested. For example, give a vector field $V\in C^k(M,T\mathbb{R}^3 )$, the
function defined by $V\cdot \nu_i$ for $x\in \Sigma _i$ is not a $C^k$ function
defined above. Actually it isn't well-defined on $M$ since on $\Gamma$, the value
will depend on $i$. So we define some large function spaces as following.
\begin{definition}
	We say a function $f(x):\bigcup _{i=1}^q \Sigma _i\rightarrow \mathbb{R}$ is in
	a Sobolev space $W^{k,p}(M)$ for $1\le p\le \infty$ if each restriction $f|_{\Sigma _i}$ is in 
	$W^{k,p}(\Sigma _i)$ for each $i$. 
	\label{def:Wkp_space}
\end{definition}

Similarly, we can define the $L^p$ space as $L^p(M)$ and continuous function
space $C^k(M)$. Usually, we will write $H^k(M)=W^{k,2}(M)$ to denote it as
Hilbert space.


By our definition, we do not impose any condition along $\Gamma $ 
for $f\in C^k(M) $. In general, we still wish
our function can also be extended to a suitable vector field on $M$ at least. 
So we say
$f\in C^k(M)$ satisfies \textit{compatible condition} if there exists
a $C^k$ vector field $W$ along $\Gamma$ (i.e. $W\in C^k(\Gamma,T\mathbb{R}^3
)$), such that
\begin{equation}
	\label{eq:ckcompactible_condition}
	f_i(x)= W(p ^{-1}_i(x))\cdot \nu_i(x) \text{ for }x\in \partial \Sigma
	_i,1\le i\le q
\end{equation}
where $f_i=f|_{\Sigma _i}$. 

Note that by Trace Theorem, if $f\in W^{1,p}(M)$ for some $1\le p\le \infty$,
then for any $1\le i\le q$, the function $f_i$ can be restricted to the boundary
$\partial \Sigma _i$ in the $L^p(\partial \Sigma _i)$ sense.

So we can say $f\in W^{1,p}(M)$ satisfies \textit{compatible condition} if there
is a $L^p_{\text{loc}}(\Gamma)$ vector field $W$ along $\Gamma$, such that
\begin{equation}
	\label{eq:compactible_condition}
	f_i|_{\partial \Sigma _i}(x)=W(p ^{-1}_i(x))\cdot \nu_i(x)\text{ for }
	\mathcal{H}^{1}
	\text{-a.e. }
	x\in \partial \Sigma _i, 1\le i\le q.
\end{equation}

Sometime we will write (\ref{eq:ckcompactible_condition}),(\ref{eq:compactible_condition}) as $f_i=W\cdot \nu_i$
for short.

Clearly, the function defined by $V\cdot \nu_i$ is in $C^k(M)$ for $V\in C^k(M,
T\mathbb{R}^3 )$ and satisfies (\ref{eq:ckcompactible_condition}). 
Conversely, for any $f\in C^k(M)$ satisfying (\ref{eq:ckcompactible_condition}), 
by definition we have $W\in
C^k(\Gamma, T\mathbb{R}^3 )$, so $f_i=W \cdot \nu_i$. For each $1\le i\le q$, we
can extend $W^\top $ on the whole $\Sigma _i$ to $\tilde{V} _i$ such that
$\tilde{V}_i$ is a $C^k$ tangential vector field on $\Sigma _i$. This is because
$W^\top $ is $C^k$ on $\Gamma$ and $\Gamma$ is smooth on $\Sigma _i$. So we can
define $V_i=\tilde{V} _i+f_i\nu_i$, which is a $C^k$ vector field on $\Sigma
_i$. So the vector field $V$ defined by $V=V_i$ on $\Sigma _i$ is in
$C^k(M,T\mathbb{R}^3 )$ and satisfies $V\cdot \nu_i=f_i$. 
\begin{remark}
	For the triple junction surface, the compatible condition has a simple form.
	For $f\in W^{1,p}(M)$, $f$ satisfies compatible condition if and only if
	\[
		f_1+f_2+f_3=0\ \ \  \mathcal{H}^{1}\text{-a.e. on }\Gamma.
	\]
\end{remark}

\begin{remark}
	All the definitions in this section can extend to arbitrary ambient
	manifolds with arbitrary dimension and codimension.
\end{remark}

\section{First and second variation of $M$}

Now we can consider the variation of $M$. We say $M_t=(\theta _1 \Sigma
_{1t},\cdots , \theta _q \Sigma _{qt};\Gamma_t), t\in (-\varepsilon ,\varepsilon )$ (considered as immersion) is a
$C^k$ variation of $M$ if each $\Sigma _{it}$ is a $C^k$ variation of
$\Sigma _i$ up to boundary and $\Gamma_t$ is a $C^k$ variation of $\Gamma$.
Of course, we can write this variation as one-parameter family of immersion
$\varphi_t(x):=\varphi (t,x):(-\varepsilon ,\varepsilon )\times M\rightarrow
\mathbb{R}^3 $ such that for each $t$, $M_t=\varphi_t(M)$ is a multiple junction
surface and restrict on each $\Sigma _i$ the variation $\varphi_t$ is $C^k$. 

For each $C^k$ variation $\varphi _t:M\rightarrow \mathbb{R}^3 $, there is an
associated vector field $V(x):M\rightarrow T_x\mathbb{R}^3 $, which is $C^k$ on $\Sigma _i$ for each $1\le i\le q$. 

Let $U\in M$ be a open subset in $M$ such that $\overline{U}$ is compact. Suppose we have a $C^k$ variation for $M$ with associated vector field
$V(x)=\frac{\partial \varphi_t (x)}{\partial t}$ with compact support in $U$. We can
define the first variation of the area of $M$ in $U$, which is given by
\begin{align*}
	\left.\frac{\text{d}}{\text{d}t}\right|_{t=0}\left|\varphi _t(U)\right|={}&
	\int_{M\cap U} \text{div}
	_{T_xM}V(x)d \|V_M\|(x)\\
		={}&\sum_{i=1}^{q} \int_{\Sigma _i \cap U} \text{div}
	_{T_x\Sigma _i} V(x)\theta _i d \mu_{\Sigma _i}(x) \\
	={}& \sum_{i=1}^{q} -\int_{\Sigma _i \cap U} V \cdot \boldsymbol{H}_{
	i} \theta_i d\mu_{\Sigma _i}+\sum_{i=1}^{q} \int_{\Gamma} V\cdot \tau_i
	\theta _i d\mu_\Gamma
\end{align*}
where $\text{div} _PV=D _{e_1}V \cdot e_1+ D _{e_2}V \cdot e_2$ for any
orthonormal basis $e_1,e_2$ of the plane $P$. 
The $\mu _{\Sigma _i}$ is the area
measure on $\Sigma _i$. $\|V_M\|$ is the weight measure of $V_M$. 

We say $M$ is \textit{stationary} in $U$ if for any such variation, we have
$\frac{\text{d}}{\text{d}t}\left|\varphi_t(U)\right|=0$.

Note that every $C^k$ vector field on $M$ will give a $C^k$ variation of $M$. So $M$ is stationary in $U$ if and only if $\boldsymbol{H}_{i}
=0$ on each $\Sigma _i\cap U$ and
on $\Gamma\cap U$, we have
\[
	\sum_{i=1}^{q} \theta _i \tau_i=0.
\]
This is precisely the condition that we define the minimal multiple junction
surface.

Now we can consider the second variation of area for minimal multiple junction
surface. 
\begin{definition}
	We say a minimal triple junction surface $M$ is \textit{stable} in $U$ whose
	closure is compact
if for every variation $\varphi _t$ of $M$ in $U$, we have
\[
	\left.\frac{\text{d}^{2}}{\text{d}t^{2}}
		\right|_{t=0}\left|\varphi _t(M\cap U)\right|\ge 0.
\]
	\label{def:stable}
\end{definition}
So we say $M$ is \textit{stable} if for every $U$ with compact closure, we
always have $\left.\frac{\text{d}^{2}}{\text{d}t^{2}}\right|_{t=0}
	\left|\varphi _t(M\cap
U)\right|\ge 0$ for any variation $\varphi _t$ in $U$.

%
%
The remaining part of this section is to deduce the stability operator
(\ref{eq:Stability}).

\begin{theorem}
	If $M$ is a stable complete minimal multiple junction surface in 
	$\mathbb{R}^3 $. Then for any $\phi \in C^k(M)$ satisfying 
	(\ref{eq:ckcompactible_condition}) with compact support, we have	
	\begin{equation}
		\label{eq:stability_inequality}
			\sum_{i=1}^{q}\int_{\Sigma _i} \left( \left|\nabla_{\Sigma _i}\phi
			_i\right|^2-\left|A_i\right|^2\phi _i^2 \right)\theta _i d\mu_{\Sigma _i} -\int_{\Gamma} \phi _i^2 {\boldsymbol{H}}_\Gamma
		\cdot \tau_i \theta _id\mu_\Gamma\ge 0
	\end{equation}
where, $\nabla_{\Sigma _i}\phi _i$ denotes the gradient on $\Sigma _i$ and
$\left|A_i\right|$ denotes the norm of the second fundamental form on $\Sigma
_i$. $\boldsymbol{H}_\Gamma$ means the curvature vector of the curve $\Gamma$. 
	\label{thm:stablility}
	Moreover, it holds even for $\phi  \in  H^1(M)$ satisfying 
	(\ref{eq:compactible_condition}) with compact support. 
\end{theorem}
\begin{proof}
	Let $\phi \in C^k(M)$. Since it satisfies compatible condition, we can find
	$V\in C^k(M,T\mathbb{R}^3 )$ such that $\phi _i=V\cdot \nu_i$. So there is
	a variation $\varphi _t
	$ with compact support associated with vector field $V $, i.e.
	$V=\left.\frac{\text{d}}{\text{d}t}\right|_{t=0}\varphi _t $

	Suppose $\varphi_t  $ is supported in $U $. So by the first variation
	formula, we have
	\[
		\frac{\text{d}}{\text{d}t}\left|\varphi _t(M \cap U)\right|
		=\sum_{i=1}^{q} -\int_{\Sigma _{it}} V_t\cdot \nu_{it} H_{\Sigma _{it}} 
		\theta _i d\mu _{\Sigma _{it}}+\sum_{i=1}^{q} \int_{\Gamma_t} V_t\cdot
		\tau_{it} \theta _i d \mu_{\Gamma_t}
	\]
	where $H_{\Sigma _{it}}=\boldsymbol{H}_{\Sigma _{it}}\cdot \nu_{it} $
	and $V_t=\frac{\text{d}}{\text{d}t}\varphi _t $.
	
	So after taking
	derivative with respect to $t $ on the first variation formula, we
	have
	\begin{align}
		\left.\frac{\text{d}^{2}}{\text{d}t^{2}}\right|_{t=0}
			\left|\varphi _t(M \cap U)\right|={} & 
			\sum_{i=1}^{q} -\int_{\Sigma _i} V \cdot \nu_i
			\left( \left.\frac{\text{d}}{\text{d}t}\right|_{t=0}H_{\Sigma 
				_{it}} \right) \theta _i d \mu_{\Sigma _i}\nonumber \\
			{} & - \sum_{i=1}^{q} \int_{\Sigma _i} H_{\Sigma _i}\theta _i
			\left.\frac{\text{d}}{\text{d}t}\right|_{t=0}\left( 
				V_t\cdot \nu_{it} d \mu _{\Sigma _{it}}\right) \nonumber \\
		{}& + \sum_{i=1}^{q} \int_{\Gamma} V\cdot 
		\left.\frac{\text{d}}{\text{d}t}\right|_{t=0}\tau_{it}\theta _i
			d \mu_{\Gamma_t}\nonumber \\
		{}&+ \sum_{i=1}^{q} 
			\int_{\Gamma} \theta _i \tau_i\cdot 
			\left.\frac{\text{d}}{\text{d}t}\right|_{t=0}(V_t d\mu_{\Gamma_t})
		\label{eq:pf_second_variation}
	\end{align}

	Note that by stationary condition, we know $H_{\Sigma _i}=0 $ and
	$\sum_{i=1}^{q} \theta _i \tau_i=0 $ along $\Gamma $, so the second
	and forth terms in (\ref{eq:pf_second_variation}) vanish. Moreover, 
	we have the well known formula (cf. \cite{rosenberg1993hypersurfaces})
	\begin{equation}
		\left.\frac{\text{d}}{\text{d}t}\right|_{t=0}H _{\Sigma _{it}}=
			\Delta _{\Sigma _i} \phi _i + \left|A_i\right|^2\phi _i
		\label{eq:pf_second_Mean_Curvature}
	\end{equation}
	So actually, we only need to compute $V\cdot
	\left.\frac{\text{d}}{\text{d}t}\right|_{t=0} \tau_{it}$. For simplicity,
		we use $(\cdot)' $ to denote $\left.\frac{\text{d}}{\text{d}
			t}\right|_{t=0}(\cdot) $.

	Before computing $\tau_i' $, we need to get $\nu_i' $. Let $e_1,e_2 $
	be the orthonormal frame of $T_x\Sigma _i $
	for some $x\in \Sigma _i $. Let $e_{it}=d \varphi _t(e_i) $. 

	We can decompose $V=\phi_i \nu_i+W_i $ on $\Sigma _i $ with $W_i $ tangential
	to $\Sigma _i $.
	Note that $[e_{it},V_t]=0 $, we have
	\begin{align*}
		\nu_i'\cdot e_i={} & -\nu_i\cdot e_i'=-\nu_i \cdot
		D _{e_i}V=-e_i(\phi_i )- \nu_i \cdot D _{e_i}W\\
		={} & - e_i (\phi_i )- A _{\Sigma _i}(e_i, W)
	\end{align*}

	Suppose $\eta $ is the unit tangential vector field on $\Gamma $ and
	also write $\eta_t=d\varphi_t (\eta) $. This time we decompose $V $ as
	$V=\phi_i \nu_i+f_i \tau_i+g \eta $ where $f_i=V\cdot \tau_i $, 
	$g=V\cdot \eta $. Similarly with $\nu'_i $, we have
	\begin{align*}
		\eta'\cdot \tau_i={} & \tau_i \cdot D _{\eta}V= \phi _i
		\tau_i \cdot D _{\eta}\nu_i+\eta(f_i)+g \tau_i \cdot D_\eta \eta\\
		={} & - \phi _i A _{\Sigma _i}(\tau_i, \eta)+ \eta(f_i)
		+ g \boldsymbol{H}_\Gamma \cdot \tau_i
	\end{align*}
	
	Hence, 
	\begin{align*}
		\tau_i' \cdot V={} & (\tau_i' \cdot \nu_i) \phi _i +
		(\tau_i' \cdot \eta) g = -(\tau_i \cdot \nu_i')\phi _i-(
		\tau_i \cdot \eta')g\\
		={} & \tau_i(\phi _i)\phi _i +\phi _i A _{\Sigma _i}(\tau_i, W)
		+\phi _i g A _{\Sigma _i}(\tau_i ,\eta) - g \eta(f_i)
		-g^2 \boldsymbol{H}_\Gamma \cdot \tau_i\\
		={} & \tau_i(\phi _i)\phi _i + \phi _i f_i A_i(\tau_i, \tau_i)
		+2\phi _i g A_i(\tau_i, \eta)-g \eta(f_i)-g^2 
		\boldsymbol{H}_\Gamma \cdot \tau_i
	\end{align*}
	
	Note that we can take $V $ which is a normal vector field on
	when restricting on $\Gamma $. In this case, we have $g=0 $ on $\Gamma $
	, so we get
	\begin{align}
		\sum_{i=1}^{q} \theta _i\tau_i'\cdot V ={} & 
		\sum_{i=1}^{q} \theta _i[\tau_i(\phi _i)\phi _i +\phi _i f_i A_i(\tau_i,
		\tau_i)]
		\label{eq:pf_second_tauV}
	\end{align}
	
	Beside, if we do not assume $V $ is normal to $\Gamma$, we can still
	get (\ref{eq:pf_second_tauV}) by noting the following identity
	\[
		\phi _i A_i(\tau_i, \eta)= \phi _i D _{\eta}\tau_i\cdot \nu_i=
		D _{\eta}\tau_i \cdot V- D _\eta \tau_i \cdot (f_i \tau_i)
		- D_\eta \tau_i \cdot (g \eta)
	\]
	
	So after taking the sum over $i $ with density $\theta _i $, we can use minimal
	condition $\sum_{i=1}^{q} \tau_i \theta _i=0 $ to get same result as
	(\ref{eq:pf_second_tauV}).

	To further processed, we use each $\Sigma _i $ is a minimal surface and
	get
	\begin{align*}
		\phi _i f_i A_i(\tau_i,\tau_i)={} & -\phi _i f_i A_i(\eta,\eta)
		=- f_i D_\eta \eta \cdot (\phi _i \nu_i)\\
		={} & -f_i \boldsymbol{H}_\Gamma \cdot V + f_i^2 
		\boldsymbol{H}_\Gamma \cdot \tau_i\\
		={} & - f_i \boldsymbol{H}_\Gamma \cdot V + 
		(\left|V\right|^2-\phi _i^2-g^2)
		\boldsymbol{H}_\Gamma \cdot \tau_i
	\end{align*}
	After taking the sum of $ i$, we have
	\[
		\sum_{i=1}^{q} \theta _i\phi _i f_iA_i(\tau_i,\tau_i)=
		0+0-\sum_{i=1}^{q} \theta _i
		\phi _i^2 \boldsymbol{H}_\Gamma \cdot \tau_i -0
		=\sum_{i=1}^{q}\theta _i \phi _i^2 \boldsymbol{H}_\Gamma\cdot \tau_i
	\]
	
	Hence, combining (\ref{eq:pf_second_variation}), 
	(\ref{eq:pf_second_Mean_Curvature}),
	(\ref{eq:pf_second_tauV}), we get the second variation formula as
	\begin{align*}
		&\left.\frac{\text{d}^{2}}{\text{d}t^{2}}\right|_{t=0}
			\left|\varphi _t(M \cap U)\right|\\
		={}&
			\sum_{i=1}^{q} \int_{\Sigma _i} \left( -\phi_i \Delta _{\Sigma _i}
		\phi _i - \left|A_i\right|^2\phi _i^2 \right) \theta _i d\mu
		_{\Sigma _i}+ \sum_{i=1}^{q}\int_{\Gamma} \left( 
		\tau_i(\phi _i)\phi _i-\phi _i^2 \boldsymbol{H}_\Gamma \cdot
	\tau_i\right)\theta _i d\mu_\Gamma\\
		={}& \sum_{i=1}^{q} \left( \int_{\Sigma _i} \left( \left|
				\nabla _{\Sigma _i}\phi _i\right|^2-
			\left|A_i\right|^2\phi _i^2\right) \theta _i 
			d \mu _{\Sigma _i} - \int_{\Gamma} \phi _i^2 
			\boldsymbol{H}_\Gamma \cdot \tau_i \theta _i d\mu_\Gamma \right) 
	\end{align*}

	So this theorem is followed by the definition of the stable.

 	Note that we can approach the function in $H^1(M)$ locally
 	by smooth functions, and the compatible condition will keep hold in the 
 	trace sense. So the above inequality still holds when $\phi \in
 	H^1(M)$ satisfying (\ref{eq:compactible_condition}) with compact support.
\end{proof}

\begin{remark}
	The proof of Theorem \ref{thm:stablility} can be extended to higher
	dimensional multiple junction hypersurfaces in an arbitrary complete ambient
	manifold $N$ directly. The stability operator will have the form
	\[
		\sum_{i=1}^{q} \int_{\Sigma _i} \left[ \left|\nabla_{\Sigma _i}\phi
		_i\right|^2-\text{Ric}_N(\nu_i)\phi _i^2-\left|A_i\right|^2\phi _i^2
	\right] \theta _i -\int_{\Gamma} \phi _i^2 \left< \boldsymbol{H}_\Gamma, \tau_i \right>
	\theta _i  .
	\]
\end{remark}

\section{Functions with finite orders}
Before giving the proof of $L^p $ estimate, we need to consider some
special functional spaces on $M $ containing $\left|A_i\right| $ and 
test functions we're interested in. Specifically, at least we want
to show $\tau_i \left( \log \left|A_i\right| \right) 
\left|A_i\right|^{2p-2}  $
are locally integrable for each $p>1 $.

Let's fix a surface $\Sigma \hookrightarrow \mathbb{R}^3 
$ with smooth boundary $\Gamma $. We will assume $0<\alpha<\infty $.

\begin{definition}
	We say a non-negative function $g(x)$ on $\Sigma  $ has 
	\textit{smooth order $\alpha $ near $x_0 $} if there is a 
	conformal coordinate chart $\varphi (z):V\rightarrow U \subset \Sigma  $
	with $0 \in V\subset \mathbb{C} $ and $\varphi (0)=x_0 $ such that
	$g(z):=g(\varphi (z)) $ has form $$g(z)=h(z)\left|z\right|^\alpha $$
	where $h(z) $ is positive and smooth in $V $.
\end{definition}

Here, the conformal coordinate chart is the coordinate chart that
metric near $x_0 $ has form $\lambda^2(z) \left|dz\right|^2 $. 
We allow $x_0 $ to be on $\Gamma $ so that $V $ is a domain with smooth 
boundary in $\mathbb{C} $.

\begin{definition}
	We call a non-negative function $g $ has \textit{smooth finite
	order} on $\Sigma  $, if there exists a discrete subset  
	$P \subset \Sigma  $ such that $g $ is smooth and positive
	on $\Sigma \backslash P $ and $g $ has smooth order 
	$\alpha_x $ near $x $ for each $x\in P $. We write this 
	function space as $\tilde{C} _+(\Sigma ) $ for conveniences.
	\label{def:Finite_order_function}
\end{definition}

Similarly, we say a non-negative function $g(t) $ on $\Gamma$ has 
\textit{smooth order $\alpha $ near $x_0 $} if $g(t) $ can be written as
$g(t)=\left|t\right|^\alpha h(t)$ for some smooth positive
function $h(t) $ under some arc length parametrization with
$g(0)= x_0$. So we can define \textit{smooth finite order}
function space $\tilde{C} _+(\Gamma)$ on $\Gamma $ which
contains the functions smooth outside a discrete set $P $
and has smooth order near each point of $P $.

We have the following lemma for the relation of these two
spaces.

\begin{lemma}
	Let $\Sigma  $ be a two dimensional analytic Riemannian
	manifold with smooth boundary $\partial \Sigma  $. 
    For any
	$\tilde{g} \in \tilde{C} _+(\Gamma) $, there is an
	extension of $\tilde{g}  $ denoted by $g $ such that
	$g\in \tilde{C} _+(\Sigma ) $. Moreover, we can require
	$g $ is positive on $\Sigma \backslash \partial \Sigma  $.

	Conversely, for any $g\in \tilde{C} _+(\partial \Sigma ) $,
	the restriction of $g $ on $\Gamma $ lies in 
	$\tilde{C} _+(\Gamma) $. 
	\label{lem:Extension_finite_order}
\end{lemma}

\begin{proof}
		Let's write $P=\{x\in \Gamma: \tilde{g} (x)=0\}. 
		$ We can just focus on the extension near each $x\in P $ since $\partial 
		\Sigma $ is smooth on $\Sigma  $ and positive smooth function can
		be easily extended from $\partial \Sigma  $ to $\Sigma  $ locally and
		keep positivity. Then
		we can use partition of unity to get a global extension.

		Fix $x_0 \in P $, we choose a conformal coordinate $
		\varphi (z):V\rightarrow 
		U$ where $V, U $ are all homeomorphic to a half disk
		such that $\varphi (0)=x_0 $. WOLG, we assume
		the metric has form $\lambda^2(z)dzd\overline{z} $ 
		with $\lambda(0)=1 $ in this chart.
		Moreover, we can assume $U $ is small enough such that $\partial 
		\Sigma \cap U$ has an arc length parametrization $\gamma(t):
		(a,b)\rightarrow \Gamma$ with $\gamma(0)=x_0 $.
		
		Let's us consider the function $f (t):=\frac{\left|t
		\right|}{\left|\varphi ^{-1}
	\circ \gamma(t)\right|} $ defined in $(-\varepsilon ',
		\varepsilon'	)\backslash 0 $ for some small $\varepsilon ' $ where $\left|z
		\right|$ is the usual absolute value in $\mathbb{C} $ with respect 
		to this coordinate chart. 
		We want to show that, by define $f 
		(0)=1$, we can get a smooth function $f$ on $(-\varepsilon '',
		\varepsilon'') $. 
		
		Note that the map $\psi(t):=\varphi ^{-1}\circ \gamma(t) $	
		is smooth from $(-\varepsilon ',\varepsilon ') $ to $V\subset 
		\mathbb{C}$ with $\psi(0)=0 $, so we can expand $\psi(t) $ as
		$\psi(t)= \psi'(0)t+\psi_1(t)t^2$ for some smooth map $\psi _1 $
		near 0.

		So we have
		\[
			f(t)=\frac{\left|t\right|}{\left|\psi(t)\right|}
			=\frac{1}{\left|\psi'(0)+\psi_1(t)t\right|}=\frac{1}{
				\left|1+t \frac{\psi_1(t)}{\psi'(0)}\right|
			}	
		\]
		is smooth in $(-\varepsilon '',\varepsilon '') $ for some small
		$\varepsilon '' $ since $\left|\psi'(0)\right|=1$.

		Based on definition of $g $, we can write $g(t)=\left|t\right|
		^{\alpha}h(t)$ for some smooth function $h(t) $ near 0. 

		So since $h(t)f(t)^{\alpha} $ is smooth and positive near 0 in 
		$\Gamma $, we
		can extend it smoothly to a neighborhood of $x_0\in \Sigma  $.
		We denote this extension function as 
		$\tilde{h}(z)  $. Then we define $\tilde{g} = \left|z\right|^{\alpha}
		\tilde{h}$ near $x$. This is a local extension of $g $ near $x$ which
		is positive except at the point $x$ since
		on $\partial \Sigma  $, we have
		\[
			\tilde{g} (\gamma(t))=\left|\varphi ^{-1}
			\circ \gamma(t)\right|^{\alpha} 
			\left( \frac{\left|t\right|}{\left|
			\varphi ^{-1}\circ \gamma(t)\right|} \right) ^{\alpha}
			h(t)=h(t)\left|t\right|^{\alpha}.
		\]
		So by partition of unity, we can get a extension of $\tilde{g}  $ as 
		we want. Moreover, we can keep $\tilde{g}  $ positive on $\Sigma 
		\backslash \partial \Sigma $. 

		Another part is essentially similar to this case. 
\end{proof}

So for our multiple junction surface $M =(\theta _1\Sigma _1,
\cdots , \theta _q\Sigma _q;\Gamma)$, the norm of second fundamental
form $\left|A_i\right| $ will belong to $\tilde{C} _+(\Sigma _i) $
if $\Sigma _i $ is non-flat. This is because when $\Sigma _i $ is minimal,
Gauss map $\nu_i(x):\Sigma _i\rightarrow \mathbb{S}^2 $ will be the 
holomorphic map. So $\left|A_i\right|^2= \left|d \nu_i\right|^2 $. Hence
$\left|A_i\right|  $ will have form $\left|z\right|^k f(z) $
near each zeros of $\left|A_i\right| $ with $f(z) $ positive and smooth
for some $k \in \mathbb{Z}_+ $ in some conformal coordinate.
Hence,
\[
	|\tau(\log \left|A_i\right|)\left|A_i\right|^{2p-2}|
	\le C_1 \left|z\right|^{2p-2} +C_2 \left|z\right|^{2p-3}
\]
where $C_i $ only depends on $f $. Since 
$\int_{-\varepsilon } ^\varepsilon \left|t\right|^{2p-2}dt, \int_{-\varepsilon} 
^\varepsilon \left|t\right|^{2p-3}$ are all finite, we know 
\[
	\int_{} (p-1) \left|\tau(\log \left|A_i\right|)\right|
	\left|A_i\right|^{2p-2} \left|\phi _i\right|^{2p}d\mu_\Gamma
\]
is locally integrable for each $\phi_i\in L^\infty_{\text{loc} }(\Gamma) $
and $p>1 $.

Moreover, we also have $\left|A\right|^\alpha\in L^\infty _{\text{loc} }
(M) \cap H^1_{\text{loc} }(M)$ by expand the gradient of $\left|A_i\right| $
near its zeros.

\section{$L^p$ estimates for the multiple junction surfaces}

In this section, we will prove the Theorem \ref{LpEstimate}. 

For convenience, we use the following notation. For any $\phi \in
H^1(M)$ with compact support, we write
\[
	\int_{\Sigma } \phi  := \sum_{i=1}^{q} \int_{\Sigma _i}  \phi _i \theta _i d\mu_{\Sigma _i} 
\]
and
\[
	\int_{\Gamma} \phi := \sum_{i=1}^{q} \int_{\Gamma} \phi _i \theta _i d\mu_\Gamma  
\]
for the integration on $M=(\theta _1\Sigma _1, \cdots ,\theta _q
\Sigma _q;\Gamma)$.

So the stability inequality can be written as

\[
	\int_{\Sigma } \left|\nabla _{\Sigma}\phi \right|^2 -\left|A\right|^2\phi
	^2-\int_{\Gamma} \phi ^2 \boldsymbol{H}_\Gamma \cdot\tau \ge 0
\]
for $\phi \in H^1(M)$ with compact support satisfying (\ref{eq:compactible_condition}). 
\begin{theorem}
	Suppose $M=(\theta _1\Sigma _1,\cdots ,\theta _q\Sigma _q;\Gamma)$ is a
	minimal multiple junction surface. Assume $M$ is stable and complete. 
	Let $\phi \in H^1(M)\cap L^\infty(M)$ such that
	$\text{sign}(\phi )\left|A\right|^{p-1}\left|\phi \right|^p$
	satisfies compatible condition (\ref{eq:compactible_condition}). Then
	\begin{align}
		\int_{\Sigma } \left|A\right|^{2p}\left|\phi \right|^{2p}  \le{} & 
		C_1\int_{\Sigma } \left|A\right|^{2p-2}\left|\phi \right|^{2p-2}
		\left|\nabla_{\Sigma }\phi \right|^2\nonumber \\
		{} & +\int_{\Gamma} \left[ \frac{p-1}{2}\left|\tau(\log \left|A\right|
		)\right|-\boldsymbol{H}_\Gamma\cdot\tau \right]  \left|A\right|^{2p-2}
		\left|\phi \right|^{2p}.
		\label{eq:LpEstmate}
	\end{align}

	Moreover, if $\phi \in W^{1,2p}(M) \cap L^\infty(M)$, we also have
	\begin{align}
		\int_{\Sigma } \left|A\right|^{2p}\left|\phi \right|^{2p}\le
		C_1'\int_{\Sigma } \left|\nabla _{\Sigma }\phi \right|^{2p}  
		+C_2' \int_{\Gamma} \left[ (p-1)|\tau(\log
		\left|A\right|)|-\boldsymbol{H}_\Gamma\cdot\tau \right]
		\left|A\right|^{2p-2}\left|\phi \right|^{2p} .
		\label{eq:Lp_Estimate2}
	\end{align}

	Here, we assume $1<p<\frac{5}{4}$, and $C_1,C_1',C_2'$ 
	will only depend on $M$, They do not depend
	on $p$. 
	\label{thm:Lp_Estimate}
\end{theorem}

Note that if the $\Sigma _i $ is flat, we can define 
\[
	\int_{\Gamma} \left|\tau_i (\log \left|A_i\right|)\right|
	\left|A_i\right|^{2p-2}\left|\phi _i\right|^{2p}=0
\]
So the right hand side of (\ref{eq:LpEstmate}) will always well-defined as
we want.

\begin{remark}
	Although the $L^p $ estimate (\ref{eq:Lp_Estimate2}) is the one appearing
	the original paper \cite{schoen1975curvature}, we still need the
	slightly stronger version one like (\ref{eq:LpEstmate}) in the 
	later application since the condition $\phi \in W^{1,2p}(M)
	\cap L^\infty(M)$ is not always satisfied based our choice of
	functions.
\end{remark}

\begin{proof}[Proof of Theorem \ref{thm:Lp_Estimate}]
	Let first consider the case that every $\Sigma _i$ is non-flat. This means
	every $\left|A_i\right|$ has only isolated zeros on $\Sigma _i$.

	We suppose $\phi \in L^\infty_{}(M)\cap H^1_{}(M)$ with compact
	support. So by H\"older's inequality, we know $\left|A\right|^{p-1}\phi 
	\in L^\infty_{}(M)\cap H^1_{}(M)$.
	Replacing $\phi $ by $\left|A\right|^{p-1}\phi $ in the stability inequality,
	we have
	\begin{multline}
		\int_{\Sigma } \left|A\right|^{2p}\phi ^2\le (p-1) ^2\int_{\Sigma }
		\left|A\right|^{2p-4}\left|\nabla_{\Sigma }\left|A\right|\right|^2\phi
		^2+\int_{\Sigma } \left|A\right|^{2p-2}\left|\nabla_{\Sigma }\phi
		\right|^2\\
		+2(p-1)\int_{\Sigma } \left|A\right|^{2p-3}\phi \nabla_{\Sigma
		}\left|A\right|\cdot \nabla_{\Sigma }\phi - \int_{\Gamma}
		\left|A\right|^{2p-2}\phi ^2\boldsymbol{H}_\Gamma\cdot \tau  .
		\label{eq:proof_Lp1}
	\end{multline}
	
	Note that the compatible condition for $\phi $ is the condition that
	$\left|A\right|^{p-1}\phi$ satisfies (\ref{eq:compactible_condition}). 

	On the minimal surface $\Sigma _i$, we have Simon's identity (see
	\cite{colding2011course} for example), we have
	\[
		\left|A_i\right|\Delta _{\Sigma
		_i}\left|A_i\right|+\left|A_i\right|^4= \left|\nabla _{\Sigma
	_i}\left|A_i\right|\right|^2
	\]
	where $\Delta _{\Sigma _i}$ is the Laplacian operator on $\Sigma _i$. 
	
	Multiplying $\left|A_i\right|^{2p-4}\phi_i ^2$ to the both side of Simon's identity
	and integrating by part, we have
	\begin{multline}
		\int_{\Sigma } \left|A\right|^{2p-4}\left|\nabla_{\Sigma
		}\left|A\right|\right|^2\phi ^2=\int_{\Sigma }
	\left|A\right|^{2p}\phi ^2-(2p-3) \int_{\Sigma }
	\left|A\right|^{2p-4}\left|\nabla_{\Sigma }\left|A\right|\right|^2\phi
	^2\\
	-2\int_{\Sigma } \left|A\right|^{2p-3}\phi  \nabla_{\Sigma }\phi
	\cdot \nabla_{\Sigma }\left|A\right|+\int_{\Gamma}
	\left|A\right|^{2p-3}\tau(\left|A\right|)\phi ^2  .
	\end{multline}

	Note that all the terms are finite in the above identity by the properties of
	$\left|A_i\right| $.
	
	Moving the second term in the right hand side of the above identity, we have
	\begin{align}
		2(p-1)\int_{\Sigma } \left|A\right|^{2p-4}\left|\nabla_{\Sigma
		}\left|A\right|\right|^2\phi ^2={}&\int_{\Sigma } \left|A\right|^{2p}\phi
		^2-2\int_{\Sigma } \left|A\right|^{2p-3}\phi \nabla_{\Sigma }\phi
		\cdot\nabla_{\Sigma }\left|A\right|\nonumber \\
		&+\int_{\Gamma} \left|A\right|^{2p-3}\tau(\left|A\right|)\phi ^2.
		\label{eq:Lp2}
	\end{align}
	Substituting (\ref{eq:Lp2}) in (\ref{eq:proof_Lp1}) and using Cauchy
	inequality, we get
	\begin{align}
		\int_{\Sigma } \left|A\right|^{2p}\phi ^2  \le{} &
		\frac{p-1}{2}\int_{\Sigma }\left|A\right|^{2p}\phi ^2 -(p-1)\int_{\Sigma }
		\left|A\right|^{2p-3}\phi \nabla_{\Sigma }\phi \cdot\nabla_{\Sigma
		}\left|A\right|   \nonumber \\
		{} & +\frac{p-1}{2}\int_{\Gamma}
		\left|A\right|^{2p-3}\tau(\left|A\right|)\phi ^2 + \int_{\Sigma }
		\left|A\right|^{2p-2}\left|\nabla_{\Sigma }\phi \right|^2\nonumber \\
		   &+2(p-1)\int_{\Sigma } \left|A\right|^{2p-3}\phi \nabla_{\Sigma }\phi
		   \cdot \nabla_{\Sigma }\left|A\right|-\int_{\Gamma} \phi
		   ^2\left|A\right|^{2p-2}\boldsymbol{H}_\Gamma\cdot\tau\nonumber \\
		={}& \frac{p-1}{2}\int_{\Sigma } \left|A\right|^{2p}\phi
		^2+(p-1)\int_{\Sigma } \left|A\right|^{2p-3}\phi  \nabla_{\Sigma }\phi
		\cdot \nabla_{\Sigma }\left|A\right|\nonumber \\
		   &+\int_{\Sigma } \left|A\right|^{2p-2}\left|\nabla_{\Sigma }\phi
		   \right|^2+\frac{p-1}{2}\int_{\Gamma}
		   \left|A\right|^{2p-2}\tau(\log(\left|A\right|))\phi ^2\nonumber \\
		   &-\int_{\Gamma}\phi ^2\left|A\right|^{2p-2}\boldsymbol{H}_\Gamma\cdot\tau\nonumber
		   \\
		\le{}&(p-1)\int_{\Sigma } \left|A\right|^{2p}\phi ^2 +\left(
		1+\frac{p-1}{2} \right) \int_{\Sigma }
		\left|A\right|^{2p-2}\left|\nabla_{\Sigma }\phi \right|^2\nonumber \\
			 &+\frac{p-1}{2}\int_{\Gamma} \left|A\right|^{2p-2}\phi ^2
			 \tau(\log(\left|A\right|))-\int_{\Gamma} \left|A\right|^{2p-2}\phi
			 ^2\boldsymbol{H}_\Gamma\cdot\tau.
	\end{align}
	
	So for $1<p<\frac{3}{2}$, we have
	\begin{align}
		\int_{\Sigma } \left|A\right|^{2p}\phi ^2  \le{} & 3\int_{\Sigma }
		\left|A\right|^{2p-2}\left|\nabla_{\Sigma }\phi \right|^2\nonumber   \\
		 {} & +\frac{p-1}{2}\int_{\Gamma} \left|A\right|^{2p-2}\phi ^2
			 \tau(\log(\left|A\right|))-\int_{\Gamma} \left|A\right|^{2p-2}\phi
			 ^2\boldsymbol{H}_\Gamma\cdot\tau.
			 \label{eq:proof_Lp3}
	\end{align}

	This is exactly the first inequality (\ref{eq:LpEstmate})
	we want to proof.

	Now let's change to the case $\phi \in W^{1,p}(M)\cap L^\infty(M) $.
	Recall the Young's inequality that for any $x,y>0$, $a,b >1$ with
	$\frac{1}{a}+\frac{1}{b}=1$, we have
	\[
		xy\le \frac{x^a}{a}+\frac{y^b}{b}.
	\]
	We choose $a=\frac{p}{p-1},b=p$, then we have
	\[
		\left|A\right|^{2p-2}\left|\phi \right|^{2p-2}\left|\nabla_{\Sigma }\phi
		\right|^2\le \frac{p-1}{p}\left|A\right|^{2p}\left|\phi
		\right|^{2p}+\frac{1}{p}\left|\nabla _{\Sigma }\phi \right|^{2p}.
	\]
	So after replacing $\phi $ by $\text{sign}(\phi )\left|\phi \right|^p$ in
	(\ref{eq:proof_Lp1}), we have
	\begin{align}
		&\left( 1-\frac{3(p-1)}{p} \right) \int_{\Sigma }
		\left|A\right|^{2p}\left|\phi \right|^{2p}\nonumber \\
		 \le{} & \frac{3}{p}\int_{\Sigma } \left|\nabla_{\Sigma }\phi
		 \right|^{2p}+\frac{p-1}{2}\int_{\Gamma} \left|A\right|^{2p-2}\left|\phi
		 \right|^{2p}\tau(\log(\left|A\right|))-\int_{\Gamma}
		 \left|A\right|^{2p-2}\left|\phi \right|^{2p}\boldsymbol{H}_\Gamma\cdot\tau  .
	\end{align}

	This time, we've assumed $\phi \in W^{1,p}(M)\cap L^\infty(M)$ with
	compact support and the
	compatible condition for $\phi$ is $\text{sign}(\phi )\left|A\right|^{p-1}
	\left|\phi \right|^{p}$ will satisfy (\ref{eq:compactible_condition}). Note
	that $\text{sign}(\phi )\left|\phi \right|^p\in W^{1,2}(M)\cap L^\infty(M)$
	, so the replacement is valid.

	So for $1<p<\frac{5}{4}$, we have
	\begin{align}
		  {} & \int_{\Sigma } \left|A\right|^{2p}\left|\phi \right|^{2p}  \nonumber \\
		  \le{} & 12 \int_{\Sigma } \left|\nabla_{\Sigma }\phi
		  \right|^{2p}+2(p-1) \int_{\Gamma} \left|A\right|^{2p-2}\left|\phi
		  \right|^{2p}\tau(\log(\left|A\right|))-4  \int_{\Gamma}
		  \left|A\right|^{2p-2}\left|\phi \right|^{2p}\boldsymbol{H}_\Gamma\cdot
		  \tau\nonumber \\
		  \le{}& 12 \int_{\Sigma } \left|\nabla_{\Sigma }\phi
		  \right|^{2p}+4\int_{\Gamma} \left[
		  (p-1)\left|\tau(\log(\left|A\right|))\right|-\boldsymbol{H}_\Gamma\cdot \tau
	  \right] \left|A\right|^{2p-2}\left|\phi \right|^{2p}  .
	  \label{eq:proof_Lp4}
	\end{align}

	This is exactly what we want. 

	If it happens that several minimal surfaces in $\{\Sigma _1, \cdots ,\Sigma
	_q\} $ are flat, and we assume $\left|A_i\right|$ cannot equal to 0 on the
	support of $\phi _i$ on $\Sigma _i$ which is not flat, then after replacing
	$\phi $ by $\left|A\right|^{p-1}\phi $, $\left|A_i\right|^{p-1}\phi _i$ will
	vanish on $\Sigma_i$ which is flat. So all the integration will still make
	sense if we just drop the terms integrated on $\Sigma _i$ which is flat and
	all the formulas above are valid. 
	
	So this estimate still holds for the general case.
\end{proof}

\section{Proof of main theorem}

In this section, we will choose a suitable test function to get our main
theorem.

Before that, let's discuss the angle condition of $\Sigma _i$ along $\Gamma$
first.

Note that since $\partial \Sigma _i$ is smooth in $\Sigma _i$, the conormal
vector fields $\tau_i$ is a smooth normal vector field along $\Gamma$. We say 
$M=(\theta _1\Sigma _1, \cdots ,\theta _q\Sigma _q;\Gamma)$ 
has \textit{equilibrium angles} along $\Gamma$ if for all $1\le i, j\le q	$, the
angles between $\tau_i, \tau_j$ are constants along $\Gamma$.

So if $M $ has equilibrium angles along $\Gamma $, we can choose a smooth
normal vector fields $W $ along $\Gamma $ such that $W $ has the constant
angle with $\tau_i $ and unit length along $\Gamma $. 
After choosing a orientation on the
normal vector field, we can write the angle between $\tau_i,W $ as $\alpha _i $,
which is a constant function on $\Gamma $. By disturbing $W $ a bit if
necessary, we can assume $\alpha _i\in (0,\frac{\pi}{2})
\cup (\frac{\pi}{2},\pi)\cup (\pi,\frac{3\pi}{2}) \cup (\frac{3\pi}{2},2\pi)
$ for all
$1\le i\le q $. We denote this normal vector field as $W_0 $ and we will use
it to construct our test function. Note that we have $W_0\cdot \nu_i=\sin 
\alpha_i$ under some suitable orientation of $\Sigma _i $,
which is non-zero and constant on $\Gamma $.
By our choice of $\alpha_i $, we know that $\cos \alpha_i $ is non-zero for
each $i $.

\begin{theorem}
	Let $M=(\theta _1\Sigma _i, \cdots ,\theta _q\Sigma _q;\Gamma)$
	be a minimal multiple junction surface in $\mathbb{R}^3 $. We assume
	$M$ is complete, stable and has quadratic area growth. Furthermore, 
	we assume $\Gamma $ is compact and has equilibrium angles along 
	$\Gamma $. Then each $\Sigma _i$ is flat.
	\label{thm:1MainTheorem}
\end{theorem}
\begin{proof}
	First, let's define a smooth cutoff function $\eta(t) $ on $\mathbb{R}  $ by
	\[
		\eta(t)=
		\begin{cases}
		1, & t\le 1\\
		0, & t\ge 2
		\end{cases}
	\]
	such that $\eta(t) $ is a monotonically decreasing on $\mathbb{R}  $ and
	$\left|\eta'(t)\right|<2 $.

	\noindent\textbf{First Case: None of $\Sigma _i $ is flat}.

	Let's write the $L^p $ estimate 
	(\ref{eq:LpEstmate}) with following notation
	\[
		\int_{\Sigma } \left|A\right|^{2p}\left|\phi \right|^{2p
		}\le C_1 \text{I}+ \text{II}- \text{III} 
	\]
	where
	\begin{align}
		\text{I}={} & \int_{\Sigma } \left|A\right|^{2p-2}\left|\phi \right|
		^{2p}\left|\nabla_{\Sigma }\phi \right|^2\nonumber \\
		\text{II}={} & \int_{\Gamma}
		\frac{p-1}{2}\left|\tau(\log \left|A\right|)\right|
		\left|A\right|^{2p-2}\left|\phi \right|^{2p}\nonumber \\
		\text{III}={}& \int_{\Gamma} 
		\boldsymbol{H}_\Gamma\cdot \tau 
		\left|A\right|^{2p-2}\left|\phi \right|^{2p}
		\label{eq:pf_def_I_II_III}
	\end{align}
	We will estimate these three terms one by one after choosing a 
	suitable test function.

	We use $T_r(\Gamma) $ to denote the tubular neighborhood of
	$\Gamma $, i.e. we define
	\[
		T_r(\Gamma):= \{x\in M:d_\Gamma(x)< r\} .
	\]
	
	For simplicity, we assume $\left|A\right| $ has no zeros in $
	\overline{T_2(\Gamma)}\backslash \Gamma$.
	Otherwise, we can do a rescaling of $M $ if necessary.

	Define the cutoff function $\rho_r $ on $M $ by
	\[
		\rho_r(x)=\eta \left(\frac{d_\Gamma(x)}{r}\right).
	\]
	
	So $\rho_r $ will has support in $T_{2r}(\Gamma) $ and equal to 1 in
	$T_r(\Gamma) $ and $\left|\nabla_{\Sigma }\rho_r\right|<\frac{2}{r} $.
	We will write $\rho(x):=\rho_1(x) $. Define $c_i=W_0\cdot \nu_i
	=\sin \alpha_i $.

	Now let's define a function $g_i $ on $\Gamma $ by
	\[
		g_i(x)=\prod _{j=1,\cdots ,q, j\neq i}\left|A_j\right|.
	\]

	Since $\left|A_i\right||_{\Gamma}\in \tilde{C} _+(\Gamma) $, 
	$g_i(x)\in \tilde{C} _+(\Gamma) $, we can extend $g_i(x) $ to
	the whole $\Sigma _i $ such that $g_i\in \tilde{C} _+(\Sigma _i) $ and
	positive on $\Sigma _i\backslash \partial \Sigma _i $ by 
	Lemma \ref{lem:Extension_finite_order}.

	Then we choose our $\phi  $ on $M $ as
	\[
		\phi _i=\text{sign}(c_i)\left|c_i\right|^{\frac{1}{p}}
		\left( \rho g_i^{\frac{p-1}{p}}+\rho_r-\rho  \right) 
	\]
	for some $r>2 $. 

	At first, we note that $\phi  $ will satisfies the compatible condition
	(\ref{eq:compactible_condition}) since on $\Gamma $, $\phi _i=
	\text{sign}(c_i)\left|c_i\right|^{\frac{1}{p}}g_i^{\frac{p-1}{p}}$, and
	\begin{align*}
		\text{sign}(\phi _i)\left|A_i\right|^{p-1}\left|\phi _i\right|^{p}=
		c_i \prod _{i=1}^q \left|A_i\right|^{p-1}=\left( \prod _{i=1}^q
		\left|A_i\right|^{p-1}\right)W_0 \cdot \nu_i.
	\end{align*}
	
	Note that $\prod _{i=1}^q \left|A_i\right|\in L_{\text{loc}}^\infty(
	\Gamma
	) $ by
	Trace Theorem (or just by the properties of functions in $\tilde{C} _+(
	\Gamma
	) $). 

	Now let's check $\phi \in H^1(M) $. Note that $g_i^{\frac{p-1}{p}} $ is
	either smooth or has form $f(z)\left|z\right|^{\frac{k(p-1)}{p}} $ near 
	an arbitrary point in $\Sigma _i $ for some smooth function $f $ in
	conformal coordinate, and $\left|z\right|^{\frac{k(p-1)}{p}} \in H^{1}
	_{\text{loc}}(M)$, so $g_i^{\frac{p-1}{p}}\in H^1_{\text{loc}}(M) $. 
	Note that $\rho,\rho_r\in W^{1,\infty}(M) $ since they are Lipschitz 
	functions with compact support, we get $\phi \in H^1(M) $ by H\"older's
	inequality.
	
	So by Theorem \ref{thm:Lp_Estimate}, we can put our $\phi  $ in the
	estimate (\ref{eq:LpEstmate}). The goal of the following proof is to make
	the terms I, II and III small enough with relatively
	large $r $ by choosing $p$ very close to $1 $ and some suitable $\phi  $.

	Now let's fix some $\varepsilon >0 $ and some $r_0>2 $ from now on. \\

	\noindent\textbf{Estimation of III.}

	Right now we do not know the sign of $\text{III} $. But if $\text{III}<0 $,
	we can rotate $W_0 $ by 90 degrees in normal bundle to get a new vector
field $\tilde{W} _0 $ along $\Gamma $. So $\tilde{c} _i:=
	\tilde{W_0} \cdot \nu_i$. Then we can define the new test function
	$\tilde{\phi }  $
	\[
		\tilde{\phi } _i= \text{sign}(\tilde{c_i} ) \left|\tilde{c} _i\right|
		^{\frac{1}{p}} \left( \rho g_i^{\frac{p-1}{p}}+\rho_r-\rho \right) .
	\]

	Along $\Gamma $, we have
	\[
		\left|\phi _i\right|=\left|c_i\right|^{\frac{1}{p}}g_i^{\frac{p-1}{p}}
		,\ \ \ 
		|\tilde{\phi } _i|=|\tilde{c_i} |^{\frac{1}{p}}
		g_i^{\frac{p-1}{p}}.
	\]

	If we define $g=\prod _{i=1}^q \left|A_i\right| $ along $\Gamma $, we have
	\begin{align}
		\sum_{i=1}^{q}\boldsymbol{H}_\Gamma\cdot \tau_i \theta _i
		 \left|A_i\right|^{2p-2}|\tilde{\phi _i} |^{2p}
		={} & \sum_{i=1}^{q} \boldsymbol{H}_\Gamma\cdot \tau_i \theta _i \tilde{c_i} ^2
		g^{2p-2}\nonumber \\
		 ={} & g^{2p-2}\sum_{i=1}^{q} \boldsymbol{H}_\Gamma\cdot\tau_i \theta _i 
		 \left( \tilde{W} _0\cdot \nu_i \right) ^2\nonumber \\
		 ={} & g^{2p-2}\sum_{i=1}^{q}\boldsymbol{H}_\Gamma\cdot\tau_i \theta _i
		 \left( W_0 \cdot \tau_i \right) ^2\nonumber \\
		 ={} & g^{2p-2}\sum_{i=1}^{q} \boldsymbol{H}_\Gamma\cdot \tau_i \theta _i
		 \left[ 1- \left( W_0\cdot \nu_i \right) ^2 \right]\nonumber  \\
		 ={} & g^{2p-2}\sum_{i=1}^{q} -\boldsymbol{H}_\Gamma \cdot\tau_i \theta _i 
		 \left( W_0\cdot\nu_i \right) ^2\nonumber \\
		 ={} & -\sum_{i=1}^{q} \boldsymbol{H}_\Gamma\cdot \tau_i \theta _i
		 \left|A_i\right|^{2p-2}\left|\phi _i\right|^{2p}.\nonumber
	\end{align}

	Here we've used the fact that $M $ is minimal along $\Gamma $.
	Hence 
	\[
		\tilde{\text{III}}:=\int_{\Gamma} \boldsymbol{H}_\Gamma\cdot\tau\left|A\right|
		^{2p-2}|\tilde{\phi } |^{2p}=-\text{III}.
	\]
	
	So after replacing $W_0 $ by $\tilde{W} _0 $, we can get 
	\begin{equation}\text{III}>0.
	\label{eq:pf_III_es}\end{equation}\\

	\noindent\textbf{Estimation of II.}

	Note that $\Gamma $ is compact, $\left|A\right| ^{2p-2}$ and $\left|\phi 
	\right|^{2p}$ are all bounded on $\Gamma$ uniformly
	with respect to $p\in (1,\frac{5}{4}) $. Since the integration 
	$\int_{\Gamma} \left|\tau(\log \left|A\right|)\right|  $ is finite by
	the property of smooth finite order functions, the integration
	\[
		\int_{\Gamma} \left|\tau(\log \left|A\right|)\right|
		\left|A\right|^{2p-2}\left|\phi \right|^{2p}< \infty.
	\]

	Hence we can choose $p_1>1 $ very close to 1 such that for every 
	$p\in (1,p_1) $, we have \begin{equation}\text{II}<
	\varepsilon \label{eq:pf_II_es}.\end{equation}\\

	\noindent\textbf{Estimation of I.}

	
	The trick part for estimating I is the points that $\phi  $ fails to be
	in $W^{1,p}(M) $ nearby, which are the zeros of $\phi  $. Denote
	$P_i=\{x\in \Sigma _i: g_i(x)=0\}  $. So by the definition of $g_i $, we know
	$P_i \subset \partial \Sigma _i $. Choose $x \in P_i $, 
	let's consider the integration
	\[
		\text{I}_{x,\delta }:=
		\int_{\Sigma _i \cap B^M_\delta (x)} \left|A_i\right|^{2p-2}\left|\phi_i 
		\right|^{2p-2}\left|\nabla_{\Sigma _i}\phi _i\right|^2
	\]
	for $\delta <1 $.

	In $\Sigma _i\cap 
	B^M_\delta (x) $, we have $\left|\phi _i\right|=\left|c_i\right|
	^{\frac{1}{p}}g_i ^{\frac{p-1}{p}}$. Still we work at conformal
	coordinate near $x$ and we choose $\delta$ small enough to make sure
	$\Sigma _i\cap  B_\delta ^M(x) $ is in this coordinate chart.
	Then $g_i  $ has form $g_i(z)=f(z)\left|z\right|^l $ for some $f(z) $
	positive and smooth near $x$ and $l\in \mathbb{Z}_+ $.

	We compute
	\begin{align}
		\left|c_i\right|^{-\frac{2}{p}}\left|\nabla_{\Sigma _i}\phi _i\right| ^2
		={} & \left( \frac{p-1}{p} \right) ^2g_i^{-\frac{2}{p}} \left|
		\nabla_{\Sigma _i}g_i\right|^2\nonumber  \\
		={} & \left( \frac{p-1}{p} \right) ^2 f^{-\frac{2}{p}}\left|z\right|
		^{-\frac{2l}{p}}\left|\left|z\right|^l\nabla_{\Sigma _i}f+
		lf\left|z\right|^{l-1}\nabla_{\Sigma _i}\left|z\right|\right|^2\nonumber 
		\\
		\le{} & 2 \left( \frac{p-1}{p} \right) ^2 f^{-\frac{2}{p}}
		\left( \left|z\right|^{2l-\frac{2l}{p}}\left|
		\nabla_{\Sigma _i}f\right|^2+l^2f^2\left|z\right|
		^{2l-2-\frac{2l}{p}}\left|\nabla_{\Sigma _i}
		\left|z\right|\right|^2\right).
		\label{eq:pf_gd_1}
	\end{align}
	
	Note that $\left|\nabla_{\Sigma }\left|z\right|\right| $ might not equal to
	1. Nerveless, it is bounded near $x$. Since $f $ is smooth
	and positive
	so it has lower bound near $x$, we know the term $f^{-\frac{2}{p}}
	\left|z\right|^{2l-\frac{2l}{p}}\left|\nabla_{\Sigma _i}f\right|^2$ is 
	bounded near $x$. Hence the integration on this term can be arbitrary
	small by choose $\delta$ small enough.

	For the term $l^2f^{2-\frac{2}{p}}\left|z\right|
	^{2l-2-\frac{2l}{p}}\left|\nabla_{\Sigma _i}\left|z
	\right|\right|^2$, if we assume the metric in this coordinate is 
	$\lambda(z)^2dz d\overline{z} $ and 
	$l^2f^{2-\frac{2}{p}}\left|\nabla_{\Sigma _i}\left|z\right|\right|\lambda^2$
	is bounded by $C $, then
	\begin{align}
		\int_{\Sigma _i\cap B_\delta ^M(x)}
		\!\!\!l^2f^{2-\frac{2}{p}}\left|z\right|
		^{2l-2-\frac{2l}{p}}\left|\nabla_{\Sigma _i}\left|z\right|\right|^2
		\le{} & \int_{|z|<\delta \text{ and }z\in \Sigma _i}
		\!\!\!\!C \left|z\right|
		^{2l-2-\frac{2l}{p}}\frac{\sqrt{-1}}{2}dz \wedge d\overline{z}\nonumber
		\\
		 \le{} & \frac{2\pi C}{2l-\frac{2l}{p}}\delta ^{2l-\frac{2l}{p}}\nonumber \\
		 ={}& \frac{\pi C p}{l(p-1)}\delta ^{2l-\frac{2l}{p}}.
		 \label{eq:pf_gd_2}
	\end{align}
	Combining (\ref{eq:pf_gd_1}),(\ref{eq:pf_gd_2}), and noting $\left|A_i\right|
	^{2p-2}\left|\phi _i\right|^{2p-2}$ is bounded near
	$x$, we get
	\[
		\text{I}_{x,\delta_x }\le C \left( \frac{p-1}{p} \right) ^2\left( \varepsilon '
		+ \frac{\pi Cp}{l(p-1)}\delta ^{2l-\frac{2l}{p}}\right)
		\le 2C \left( \frac{p-1}{p} \right) ^2 \varepsilon '
	\]
	for some $\delta _x $ small enough.
	
	Here, the constance $C $ does not depend on $p$. So we can choose $p_x
	\in (1,p_0)$ small enough such that if $p\in (1,p_x) $, then
	\[
		\text{I}_{x,\delta_x }<\varepsilon'
	\]
	for any given $\varepsilon '>0 $. Note that the set $P_i $ is a finite set
	for each $1\le i\le q $, so we can choose $p_1 = \min _{x\in \cup P_i}p_x$
	and $\delta _0= \min _{x\in \cup  P_i}\delta _x $. Denote $B\in M $ as
	\[
		B:= \bigcup _{i=1}^q \bigcup _{x\in P_i}
		\left( \Sigma _i \cap B_{\delta _0}(x) \right).
	\]
	Then
	\[
		\int_{B}\left|A\right|^{2p-2}\left|\phi \right|^{2p-2}
		\left| \nabla_{\Sigma }\phi \right|^2\le 
		\sum_{i=1}^{q} \theta _i \sharp(P_i) \varepsilon '
	\]
	where $\sharp(P_i) $ denote the cardinality of the set $P_i $. By requiring
	$\varepsilon ' $ small, we can find $p_1 \in (1,p_0), \delta_0 >0$ such
	that
	\begin{equation}
		\text{I}'=\int_{B} \left|A\right|^{2p-2}\left|\phi \right|^{2p-2}
		\left|\nabla_{\Sigma }\phi \right|^2\le \frac{\varepsilon}{C_1}.
		\label{eq:pf_Es_B}
	\end{equation}

	Now we focus on the estimation of
	\[
		\text{I}''=\int_{\Sigma \backslash B} \left|A\right|^{2p-2}
		\left|\phi \right|^{2p-2}\left|
		\nabla_{\Sigma }\phi \right|^2.
	\]

	Clearly we have
	\begin{equation}
		\text{I}=\text{I}'+\text{I}''.
		\label{eq:pf_I_c}
	\end{equation}

	Note that right now we know $\phi  $ is positive and smooth with compact
	support in $\Sigma \backslash B $, so $\phi \in W^{1,\infty}(M) $. So we
	can apply Young's inequality to get
	\begin{equation}
		\text{I}''\le \frac{p-1}{p}\int_{\Sigma \backslash B} 
		\left|A\right|^{2p}\left|\phi 
		\right|^{2p} + \frac{1}{p} \int_{\Sigma \backslash B} 
		\left|\nabla_{\Sigma }
		\phi \right|^{2p}<\infty .
		\label{eq:I_2_Es}
	\end{equation}
	
	So we choose $p_2\in (1,p_1) $ small such that 
	\begin{equation}\frac{p_2-1}{p_2}<
	\frac{1}{4C_1}
	\label{eq:pf_p_2}
	\end{equation} 
	so that the first term in the above inequality can be absorbed
	by left hand side of (\ref{eq:LpEstmate}). For the second term in 
	(\ref{eq:I_2_Es}), we have
	\begin{align}
		\int_{\Sigma \backslash B} \left|\nabla_{\Sigma }\phi \right|^{2p
		}={} & \int_{T_2(\Gamma)\backslash B}  \left|\nabla_{\Sigma }\phi 
		\right|^{2p}+\int_{T_{2r}(\Gamma)\backslash T_r(\Gamma)} 
		\left|\nabla_{\Sigma }\phi \right|^{2p}\nonumber \\
		 =:{} & \text{I}_1+\text{I}_2.
		 \label{eq:pf_L^p}
	\end{align}

	In $T_2(\Gamma)\backslash B $, we know $\left|\phi _i\right|=
	\left|c_i\right|^{\frac{1}{p}}\rho \left( g_i^{\frac{p-1}{p}}
	-1\right)+\left|c_i\right|^{\frac{1}{p}} $. Hence
	\begin{align}
		\left|\nabla_{\Sigma _i}\phi _i\right|^{2p} ={} & 
		\left|c_i\right|^{2}\left|
		\frac{p-1}{p}g_i^{-\frac{1}{p}}\nabla_{\Sigma _i}g_i\rho+
		\nabla_{\Sigma _i}\rho \left( g_i^{\frac{p-1}{p}}-1 \right) \right|^{2p}
		\nonumber \\
		\le{}& 2^{2p}c_i^2 \left( \frac{p-1}{p} \right) ^{2p}g_i^{-2}\rho^{2p}
		\left|\nabla_{\Sigma _i}g_i\right|^{2p}\nonumber \\
			 &+2^{2p}c_i^2\left|\nabla_{\Sigma _i}\rho\right|^{2p}
			 \left|g_i^{\frac{p-1}{p}}-1\right|^{2p}\nonumber \\
		\le{}&C(p-1)^2g_i^{-2}\left|\nabla_{\Sigma _i}g_i\right|^{2p}+C
		\left|g_i^{\frac{p-1}{p}}-1\right|^{2p}.
	\end{align}

	Note that $g_i $ has positive upper and lower bound on $T_2(\Gamma)
	\backslash B$, $\left|\nabla_{\Sigma _i}g_i\right|$ has an upper bound
	on $T_2(\Gamma)\backslash B $ since it is smooth. Hence, 
	as $p\rightarrow 1^+ $, $g_i \rightarrow 1$ uniformly. 
	So we can choose $p_3\in
	(1,p_2)$ small enough to make 
	\begin{equation}
		\text{I}_1<\frac{\varepsilon}{C_1}
		\label{eq:pf_I_1}
	\end{equation}
	for any $p\in (1,p_3) $. From now on, we will fix $p\in(1,p_3) $.

	For the integration $\text{I}_2 $, we have
	\begin{align}
		\text{I}_2 ={} & \int_{T_{2r}(\Gamma)\backslash T_r(\Gamma)} 
		c^2\left|\nabla_{\Sigma}\rho_r\right|^{2p}\nonumber \\
		\le{} & C\int_{T_{2r}(\Gamma)\backslash T_r(\Gamma)}
		\frac{1}{r^{2p}}\nonumber \\
		\le{} & C r^{2-2p}\ \ \ \text{By area growth condition}.
		\label{eq:pf_L_2}
	\end{align}

	Here the constant $C $ does not depend on $p $ and $r $. Since we've fixed
	$p $, we can choose $r>r_0 $ so large, such that
	\begin{equation}
		\text{I}_2\le Cr^{2-2p}< \frac{\varepsilon}{C_1}.
		\label{eq:pf_I_2_fin}
	\end{equation}
	
	Now, let's combine the estimates (\ref{eq:pf_I_2_fin}), (\ref{eq:pf_I_1}),
	(\ref{eq:pf_L^p}), (\ref{eq:I_2_Es}), 
	(\ref{eq:pf_p_2}), (\ref{eq:pf_I_c}),
	(\ref{eq:pf_II_es}), (\ref{eq:pf_III_es}), (\ref{eq:pf_def_I_II_III})
	with (\ref{eq:LpEstmate}) to get
	\begin{equation}
		\int_{\Sigma } \left|A\right|^{2p}\left|\phi \right|^{2p}
		\le \frac{1}{4}\int_{\Sigma \backslash B} 
		\left|A\right|^{2p}\left|\phi \right|^{2p}
		+\varepsilon +\varepsilon +\varepsilon +\varepsilon .
		\label{eq:pf_Lp_est_1}
	\end{equation}

	So
	\begin{equation}
		\int_{\Sigma } \left|A\right|^{2p}\left|\phi \right|^{2p}\le
		8\varepsilon .
	\end{equation}
	
	By definition of $\phi  $, we have
	\begin{align}
		\int_{T_{r_0}(\Gamma)\backslash T_2(\Gamma)} 
		\min \{1, \left|A\right|^{4}\}c^2 \le{} &  \int_{
			T_r(\Gamma)\backslash T_2(\Gamma)
		} \left|A\right|^{2p}c^2 \nonumber \\
		 \le{} & \int_{\Sigma } \left|A\right|^{2p}\left|\phi \right|^{2p} 
		 \nonumber \\
		 \le{} & 8 \varepsilon .
	\end{align}
	
	So the left hand side of the above inequality does not depend on $r $ and
	$p $. By arbitrary choice of $\varepsilon  $ and noting $c_i\neq 0 $ on
	each $\Sigma _i $, we know actually $\left|A\right|=0 $ on $T_{r_0}(
	\Gamma
	) \backslash T_2(\Gamma)$, and thus this implies each $\Sigma _i $ is flat,
	which contradicts our assumption that each $\Sigma _i $ is non-flat.\\
	
	\noindent\textbf{Second case: One of $\Sigma _i $ is flat.}

	For the case that one of $\Sigma _i $ being flat, we suppose $\Sigma _1 $ is
	lying the plane $P $. Clearly, if there are another $\Sigma _i $, which is
	flat and different from $\Sigma  _1$, $\Sigma _i $ should be lying $P $,
	too since we've assume $\Gamma $ is compact. Again, we write it as $
	\Sigma _2$ for simplicity. So the only possible choice of
	$\Sigma _2 $ is $\Sigma _2=P\backslash (\Sigma _1 \cup \partial \Sigma _1)$.
	Note that $P $ is stable, so we can remove it from this triple junction
	surface to get the remaining one $((\theta _1-\theta _2)
	\Sigma _1, \theta _3\Sigma _3,\cdots , \theta _q\Sigma _q;\Gamma) $ or 
	$((\theta _2-\theta _1)\Sigma _2, \theta _3 \Sigma _3,\cdots 
	,\theta _q\Sigma _q;\Gamma) $ is unstable. So WOLG, we assume there are
	only one $\Sigma _i $, which we call it $\Sigma _1 $, is flat.
	
	This time we choose $W_0 = \tau_1$. The trick park is the term 
	III might not have a favorable sign. So we need
	a bit more precise estimation of the total curvature. 

	Let's use $K_i $ to denote the sectional curvature on $\Sigma _i $. 
	Following from B. White's proof (\cite{white1987complete}), we have the
	following lemma.

	\begin{lemma}
		For each $\Sigma _i $ which is non-compact with boundary $\Gamma $, we
		have
		\[
			\int_{\Gamma} - \boldsymbol{H}_\Gamma\cdot \tau_i \le \int_{\Sigma _i} -K_{\Sigma
			_i}  
		\]
		where the $K_{\Sigma _i} $ is the sectional curvature of $\Sigma _i $.
		\label{lem:Boundary_curvature_estimate}
	\end{lemma}
	\begin{proof}
		Fix a point $p_0\in \Gamma $, and we define
		\[
			B_r:= \{ p\in \Sigma _i: d_i(p,p_0)<r\} .
		\]
		
		We write $\Gamma_r=\partial B_r\backslash \Gamma $, the remaining 
		boundary part of $B_r $ except $\Gamma $.

		Note that we can choose a large $r_0 $ such that $\Gamma
		\subset B _{r_0}$. So for $r>r_0 $, we know $\Gamma $ and
		$\Gamma_r $ do not connect with each other.

		By the result of P. Hartman \cite{hartman1964geodesic}, we know
		$\Gamma_r $ is, for almost all $r $, a piecewise smooth, embedded closed
		curve in $\Sigma _i $. So we can apply Gauss-Bonnet theorem to get
		\begin{align}
			{} & \int_{B_r} K_{\Sigma _i} + \int_{\Gamma}\boldsymbol{H}_\Gamma\cdot (-\tau_i)
			+\int_{\Gamma_r} \kappa_g + 
			\sum_{  }^{} (\text{exterior angles of }\Gamma_r)\nonumber \\
			={} & 2 \pi \chi(B_r)=2\pi(2-2h(r)-c(r))
		\end{align}
		where $h(r)$ and $c(r) $ are the number of handles and the number of
		boundary components, respectively, of $B_r $ with $r>r_0 $. Here we
		also use $\kappa_g $ to denote the curvature of $\Gamma_r $ inside
		the surface $\Sigma _i $ with respect to the inner conormal vector
		field.

		Let $L(r) $ be the length of $\Gamma_r $. So by the first variation
		formula of piecewise smooth curve, we have
		\[
			L'(r)=\int_{\Gamma_r} \kappa_g+ \sum \text{exterior
			angles of $\Gamma_r $} .
		\]

		We also note that $c(r)\ge 2 $ since $\Gamma $ has at least one
		component and $\Gamma_r $ is always non-empty since we've assume 
		$\Sigma _i $ is complete and not compact for any $r>r_0 $. So
		combining with the above Gauss-Bonnet formula we've got, we can get
		\[
			-\int_{\Gamma} \boldsymbol{H}_\Gamma\cdot\tau_i+L'(r)\le - \int_{B_r} 
			K_{\Sigma _i}.
		\]
		
		Note that $L(r)>0 $ for all $r>r_0 $, we have $
		\lim_{} \sup _{r\rightarrow \infty}L'(r)\ge 0$. And since $K_{\Sigma 
		_i}\le 0 $, we can take $r\rightarrow \infty $ to get
		\[
			-\int_{\Gamma} \boldsymbol{H}_\Gamma\cdot\tau_i 
			\le -\int_{\Sigma _i} K_{\Sigma _i}  .
		\]
		This is what we want.
	\end{proof}

	Let's go back to the proof of the main theorem. Again, we choose the function
	$g_i $ on $\Gamma $ as
	\[
		g_i(x)=\prod _{j=2,\cdots ,q, j\neq i} \left|A_j\right|
	\]
	and choose our $\phi  $ on $M $ as
	\[
		\phi :=\phi_{r,p} = \text{sign}(c) \left|c\right|^{\frac{1}{p}}
		\left( \rho g ^{\frac{p-1}{p}}+\rho_r-\rho \right) 
	\]
	where $c_i = W_0\cdot \nu_i $. Here we use subscript to indicate $\phi  $
	depends on $r $ and $p $ if needed. Clearly, $\text{sign}(\phi _i)
	\left|A_i\right|^{p-1}\left|\phi _i\right|^{p}$ satisfies the compatible
	condition. 
	
	This time, we do not have a good sign for the term III, so we keep it
	in our estimates. Based on essentially same argument, we can get a similar
	estimate like (\ref{eq:pf_Lp_est_1}) as
	\begin{equation}
		\int_{\Sigma } \left|A\right|^{2p}\left|\phi \right|^{2p}
		\le \frac{1}{4} \int_{\Sigma \backslash B}\left|A\right|^{2p}
		\left|\phi \right|^{2p}+4\varepsilon - \text{III}
		\label{eq:pf_Lp_Est_5.21}
	\end{equation}
	for $p $ small enough and $r $ large enough which might depend on $p $.
	If it happens that $\text{III}\ge0 $, the previous argument shows that 
	each $\Sigma _i $ for $i=2,\cdots ,q $ is flat.

	For simplicity, we write III in the form which depends on $p $ as
	\[
		\text{III}_p :=\int_{\Gamma} \left|A\right|^{2p-2}\left|
		\phi \right| ^{2p} \boldsymbol{H}_\Gamma\cdot\tau = \int_{\Gamma} 
		g^{2p-2}c^2 \boldsymbol{H}_\Gamma\cdot\tau
	\]
	where $g= \prod _{i=2}^q \left|A_i\right| $.
	So if we write $\text{III}_0 $, we just mean
	\[
		\text{III}_0:= \int_{\Gamma} c^2 \boldsymbol{H}_\Gamma\cdot\tau .
	\]
	We note $g^{2p-2}\rightarrow 1
	$ a.e. on $\Gamma $ since the zeros of $g $ are isolated, so by Dominated 
	Convergence theorem, we have
	\[
		\int_{\Gamma} g^{2p-2}c^2\boldsymbol{H}_\Gamma\cdot \tau\rightarrow \int_{\Gamma} 
		c^2\boldsymbol{H}_\Gamma\cdot \tau
	\]
	as $p\rightarrow 1 $. This means we can choose $p $ large to make
	\begin{equation}
		\left|\text{III}_p-\text{III}_0\right|\le \frac{1}{8} \left|\text{III}_0\right|
		\label{eq:pf_IIIp_est}
	\end{equation}
	since $\text{III}_0<0 $ as we've assumed.

	Similarly, $g_i^{\frac{p-1}{p}}\rightarrow 1 $ a.e. on $\Sigma _i $ and
	$\left|A_i\right|^{2p}\rightarrow \left|A_i\right|^2 $ a.e. on 
	$\Sigma _i $, we have
	\[
		\int_{\Sigma } \left|A\right|^{2p}\left|\phi \right|^{2p}\rightarrow 
		\int_{\Sigma } \left|A\right|^2c^2\rho_r^2 
	\]
	as $p\rightarrow 1 $.
	
	Hence, for some fixed $r $, we can always choose $p $ small enough to 
	make sure
	\begin{equation}
		\left|\int_{\Sigma } 
		\left|A\right|^{2p}\left|\phi \right|^{2p}-
		\left|A\right|^2c^2\rho^2_r\right|\le \frac{1}{8}\left|
		\text{III}_0\right|.
		\label{eq:pf_Lp_pclosedto1}
	\end{equation}
	
	By our previous lemma \ref{lem:Boundary_curvature_estimate}, we have
	\[
		\int_{\Sigma } \left|A\right|^2c^2=-\int_{\Sigma } 2K_\Sigma c^2
		\ge -2\int_{\Gamma} c^2\boldsymbol{H}_\Gamma\cdot\tau=2\left|\text{III}_0\right|.
	\]

	Note $\int_{\Sigma } \left|A\right|^2c^2\rho_r^2\rightarrow 
	\int_{\Sigma } \left|A\right|^2c^2 $, so we can fix a $r_1 $ large enough
	to get
	\begin{equation}
		\int_{\Sigma } \left|A\right|^2c^2\rho_{r_1}^2\ge \frac{15}{8}
		\left|\text{III}_0\right|.
		\label{eq:pf_L2_curvature}
	\end{equation}

	Combining with (\ref{eq:pf_Lp_pclosedto1}), we have
	\[
		\int_{\Sigma } \left|A\right|^{2p}\left|\phi _{r_0,p}\right|^{2p}
		\ge \frac{7}{4} \left|\text{III}_0\right|
	\]
	for $p $ sufficient small. Note $\phi _{r,p} $ is an increasing function
	with respect to variable $r $ as $r>2 $, so we actually have 
	\begin{equation}
		\int_{\Sigma } \left|A\right|^{2p}\left|\phi _{r,p}\right|^{2p
		}\ge \frac{7}{4}\left|\text{III}_0\right| 
		\label{eq:pf_Lp_est_2}
	\end{equation}
	for all $r>r_0, 1<p<p_0 $ for some $p_0 $ sufficient closed to 1. 

	Hence, we can choose $p $ small enough again, and then choose $r >r_0$ large
	enough to make the estimate (\ref{eq:pf_Lp_Est_5.21}) holds for 
	$\varepsilon = \frac{1}{32}\left|\text{III}_0\right| $. So 
	(\ref{eq:pf_Lp_Est_5.21}) implies
	\[
		\int_{\Sigma } \left|A\right|^{2p}\left|\phi \right|^{2p}\le 
		\frac{1}{6}\left|\text{III}_0\right|-\frac{4}{3}\text{III}_p\le 
		\frac{1}{6}\left|\text{III}_0\right|+\frac{3}{2}\left|
		\text{III}_0\right|=\frac{5}{3}\left|\text{III}_0\right|
	\]
	where we've used (\ref{eq:pf_IIIp_est}). This is a contradiction with
	the estimate (\ref{eq:pf_Lp_est_2}). So it is impossible that only 
	$\Sigma _1 $ is flat. Hence, we finished our proof.
\end{proof}

\begin{theorem}
	Let $M=(\theta _1, \Sigma _1, \cdots ,\theta _q \Sigma _q; \Gamma) $ be 
	a minimal multiple junction surface in $\mathbb{R}^3  $. We assume $M $ is
	complete, stable and has quadratic area growth. Furthermore, we assume
	$\Gamma $ is a straight line and $M $ has equilibrium angles along $\Gamma$,
	then each $\Sigma _i $ is flat.
	\label{thm:Gam_straight}
\end{theorem}
\begin{proof}
	This case is much easier than the case of $\Gamma $ compact. Note that for
	any $i\neq j $, by rotating $\Sigma _i $ along $\Gamma $ for a suitable
	angle, we can make $\Sigma _i $ and $\Sigma _j $ share the same outward
	conormal of boundary. Hence by Hopf's boundary lemma, $\Sigma _i $ will
	be identical to $\Sigma _j $ after rotation.

	This says that each pieces of surface in $M $ are all isometric to each
	other. Moreover, by reflection principle, we also have $\tau_i(\left|
	A_i\right|) =0$ for each $1\le i\le q $. Hence the stability operator for
	$M $ is
	\[
		\int_{\Sigma } \left|A\right|^2\phi ^2\le \int_{\Sigma } 
		\left|\nabla_{\Sigma }\phi \right|^2
	\]
	for some $\phi =W\cdot \nu $. To apply the $L_p $ estimate, we need
	$\text{sign}(\phi )\left|A\right|^{p-1}\left|\phi \right|^{2p} $ to
	satisfy the compatible condition (\ref{eq:compactible_condition}). 
	Note that $\left|A_i\right|=\left|A_j\right| $, we only need to require
	$\text{sign}(\phi )\left|\phi \right|^{p} $ to satisfy 
	(\ref{eq:compactible_condition}).

	Now, as usual we choose $W_0 $ having constant angles with each $\tau_i $
	and make sure $c_i:= W_0\cdot \nu_i \neq 0$ for each $i $. Now we
	choose an arbitrary point $x_0 $ on $\Gamma $. Define cutoff function
	$\rho_r(x) $ support in $B_{2r}^M(x) $ which equals to 1 in $B_r^M(x) $ and
	has gradient less than $\frac{2}{r} $. 

	So we can choose our $\phi  $ as
	\[
		\phi = \text{sign}(c)\left|c\right|^{\frac{1}{p}}\rho_r.
	\]

	Hence we can apply $L^p $ estimate to get
	\[
		\int_{\Sigma } \left|A\right|^{2p}\left|\phi \right|^{2p}\le
		C_1\int_{\Sigma } \left|\nabla_{\Sigma }\phi \right|^{2p} .
	\]

	Standard argument in \cite{schoen1975curvature} will imply each $\Sigma _i $
	will be flat.
\end{proof}

\begin{remark}
	Indeed, we can remove the quadratic area growth condition in the case 
	$\Gamma $ a straight line. Actually, one can just use the result of 
	Bernstein's theorem for stable minimal surface to get this result.
\end{remark}

As a corollary, we can also get some result related to the stable capillary
minimal surface. 

\begin{corollary}
	Let $P $ be a plane in $\mathbb{R}^3  $. Then there is no 
	(oriented) stable complete minimal surface $\Sigma  $ with boundary $
	\partial \Sigma $ such that $\partial \Sigma \in P $, $\partial \Sigma  $
	compact, and $\Sigma  $ has constant angle with $P $ along $\partial 
	\Sigma $.
\end{corollary}
Here the stability of the capillary minimal surface means this 
surface is stable of capillary energy under the variation fixing
the plane $P $.

This result is an immediate result in the proof of Theorem 
\ref{thm:1MainTheorem}. The variation we've taken in the proof
of the second case is just the one fixing the plane $P $.

\section{Further questions}

As we've seen, we still leave some questions related to
minimal multiple junction surface.

The first one is, what if $\Gamma $ is a neither compact nor straight
line?
\begin{question}
	Can we relax the condition of $\Gamma $ to be compact and
	straight line, so that we still get the similar 
	Bernstein's theorem for the stable minimal multiple
	junction surface?
\end{question}
In particular, we want to know the stability property of
universal cover of $Y $-shaped bent helicoid.

Another question is, we still need the help of area growth
to get control of curvature. So we may still want
to remove this condition in some sense.

\begin{question}
	Can one get the quadratic area growth for
	stable minimal multiple junction surfaces like the
	result in \cite{colding2002estimates} in some
	sense? 
\end{question}

In \cite{colding2002estimates}, one may need the simply connected
condition to get the quadratic area growth. This is not
a problem when talking about the usual smooth surface since
we can always take a universal cover without affecting 
stability. But things get unusual especially when requiring
$\Gamma $ compact. Moreover, one may still need
careful consideration when talking about the simply connected
multiple junction surfaces.

The last question is related to the higher dimension case.
\begin{question}
	Can one get the similar Bernstein's theorem for stable
	minimal multiple junction hypersurface with the
	hypersurface dimension greater than 2?
\end{question}
As we've seen, basically we can still get the similar $L^p $
estimation of curvature like Schoen, Simon, Yau's result 
\cite{schoen1975curvature} for $2\le n\le 5 $. But
we have an additional compatible condition along $\Gamma $,
so we need $p\rightarrow 1 $ in order to get
the curvature estimation near $\Gamma $. This will force 
our surface dimension to be $2 $.

\section*{Acknowledgements}
	I would like to thank my advisor Prof. Martin Li for
	his helpful discussions and encouragement.

This work is substantially supported by a research grant from the Research Grants Council of the Hong Kong Special Administrative Region, China [Project No.: CUHK 14301319]
\bibliographystyle{unsrt}  
\bibliography{references} 

\begin{thebibliography}{10}

\bibitem{schoen1975curvature}
Richard Schoen, Leon Simon, and Shing-Tung Yau.
\newblock Curvature estimates for minimal hypersurfaces.
\newblock {\em Acta Mathematica}, 134(1):275--288, 1975.

\bibitem{colding2002estimates}
Tobias~H Colding and William~P Minicozzi.
\newblock Estimates for parametric elliptic integrands.
\newblock {\em International Mathematics Research Notices}, 2002(6):291--297,
  2002.

\bibitem{taylor1976structure}
Jean~E Taylor.
\newblock The structure of singularities in soap-bubble-like and soap-film-like
  minimal surfaces.
\newblock {\em Annals of Mathematics}, pages 489--539, 1976.

\bibitem{lawlor1996curvy}
Gary Lawlor, Frank Morgan, et~al.
\newblock Curvy slicing proves that triple junctions locally minimize area.
\newblock {\em J. Diff. Geom}, 44:514--528, 1996.

\bibitem{mese2006parameterized}
Chikako Mese and Sumio Yamada.
\newblock The parameterized steiner problem and the singular plateau problem
  via energy.
\newblock {\em Transactions of the American Mathematical Society},
  358(7):2875--2895, 2006.

\bibitem{schoen1983uniqueness}
Richard~M Schoen.
\newblock Uniqueness, symmetry, and embeddedness of minimal surfaces.
\newblock {\em Journal of Differential Geometry}, 18(4):791--809, 1983.

\bibitem{bernstein2021symmetry}
Jacob Bernstein and Francesco Maggi.
\newblock Symmetry and rigidity of minimal surfaces with plateau-like
  singularities.
\newblock {\em Archive for Rational Mechanics and Analysis}, 239(2):1177--1210,
  2021.

\bibitem{allard1972first}
William~K Allard.
\newblock On the first variation of a varifold.
\newblock {\em Annals of mathematics}, pages 417--491, 1972.

\bibitem{simon1993cylindrical}
Leon Simon et~al.
\newblock Cylindrical tangent cones and the singular set of minimal
  submanifolds.
\newblock {\em Journal of Differential Geometry}, 38(3):585--652, 1993.

\bibitem{krummel2014regularity}
Brian Krummel.
\newblock Regularity of minimal hypersurfaces with a common free boundary.
\newblock {\em Calculus of Variations and Partial Differential Equations},
  51(3-4):525--537, 2014.

\bibitem{freire2010mean}
Alexandre Freire.
\newblock Mean curvature motion of triple junctions of graphs in two
  dimensions.
\newblock {\em Communications in Partial Differential Equations},
  35(2):302--327, 2010.

\bibitem{depner2013linearized}
Daniel Depner, Harald Garcke, et~al.
\newblock Linearized stability analysis of surface diffusion for hypersurfaces
  with triple lines.
\newblock {\em Hokkaido Mathematical Journal}, 42(1):11--52, 2013.

\bibitem{depner2014mean}
Daniel Depner, Harald Garcke, and Yoshihito Kohsaka.
\newblock Mean curvature flow with triple junctions in higher space dimensions.
\newblock {\em Archive for Rational Mechanics and Analysis}, 211(1):301--334,
  2014.

\bibitem{schulze2020local}
Felix Schulze and Brian White.
\newblock A local regularity theorem for mean curvature flow with triple edges.
\newblock {\em Journal f{\"u}r die reine und angewandte Mathematik},
  2020(758):281--305, 2020.

\bibitem{mantegazza2004motion}
Carlo Mantegazza, Matteo Novaga, and Vincenzo~Maria Tortorelli.
\newblock Motion by curvature of planar networks.
\newblock {\em Annali della Scuola Normale Superiore di Pisa-Classe di
  Scienze}, 3(2):235--324, 2004.

\bibitem{bronsard1993three}
Lia Bronsard and Fernando Reitich.
\newblock On three-phase boundary motion and the singular limit of a
  vector-valued ginzburg-landau equation.
\newblock {\em Archive for Rational Mechanics and Analysis}, 124(4):355--379,
  1993.

\bibitem{ilmanen2014short}
Tom Ilmanen, Andr{\'e} Neves, and Felix Schulze.
\newblock On short time existence for the planar network flow.
\newblock {\em arXiv preprint arXiv:1407.4756}, 2014.

\bibitem{tonegawa2016blow}
Yoshihiro Tonegawa and Neshan Wickramasekera.
\newblock The blow up method for brakke flows: networks near triple junctions.
\newblock {\em Archive for Rational Mechanics and Analysis}, 221(3):1161--1222,
  2016.

\bibitem{meeks2007bending}
William~H Meeks and Matthias Weber.
\newblock Bending the helicoid.
\newblock {\em Mathematische Annalen}, 339(4):783--798, 2007.

\bibitem{rosenberg1993hypersurfaces}
Harold Rosenberg.
\newblock Hypersurfaces of constant curvature in space forms.
\newblock {\em Bull. Sci. Math}, 117(2):211--239, 1993.

\bibitem{colding2011course}
Tobias~H Colding and William~P Minicozzi.
\newblock {\em A course in minimal surfaces}, volume 121.
\newblock American Mathematical Soc., 2011.

\bibitem{white1987complete}
Brian White et~al.
\newblock Complete surfaces of finite total curvature.
\newblock {\em Journal of Differential Geometry}, 26(2):315--326, 1987.

\bibitem{hartman1964geodesic}
Philip Hartman.
\newblock Geodesic parallel coordinates in the large.
\newblock {\em American Journal of Mathematics}, 86(4):705--727, 1964.

\end{thebibliography}

\end{document}